\documentclass[11pt]{article}
\usepackage{amsmath}
\usepackage{amsfonts}
\usepackage{amssymb, amsthm}
\usepackage{color}
\usepackage[margin=1in]{geometry}
\usepackage{hyperref}
\usepackage{verbatim}
\usepackage{ulem}
\usepackage{indentfirst}
\usepackage{graphicx}
\usepackage{caption}
\usepackage{subcaption}
\usepackage{multirow}
\usepackage{thm-restate}

\newtheorem{thm}{Theorem}[subsection]
\newtheorem{theorem}[thm]{Theorem}
\newtheorem{proposition}[thm]{Proposition}
\newtheorem{lemma}[thm]{Lemma}
\newtheorem{corollary}[thm]{Corollary}
\newtheorem{conjecture}[thm]{Conjecture}

\theoremstyle{definition}
\newtheorem{defn}[thm]{Definition}

\newtheorem{eg}[thm]{Example}
\newtheorem{rmk}[thm]{Remark}

\newcommand{\mc}{\mathcal}

\allowdisplaybreaks

\title{Circular Planar Electrical Networks I: The Electrical Poset $EP_{n}$}
\author{Joshua Alman, Carl Lian, and Brandon Tran\footnote{Department of Mathematics, Massachusetts Institute of Technology. \href{mailto:jalman@mit.edu}{jalman@mit.edu}, \href{mailto:clian@math.mit.edu}{clian@math.mit.edu}, \href{mailto:btran115@mit.edu}{btran115@mit.edu}}}

\begin{document}

\maketitle

\begin{abstract} Following de Verdi\`{e}re-Gitler-Vertigan and Curtis-Ingerman-Morrow, we prove a host of new results on circular planar electrical networks. We introduce a poset $EP_{n}$ of electrical networks with $n$ boundary vertices, giving two equivalent characterizations, one combinatorial and the other topological. We then investigate various properties of the $EP_{n}$, proving that it is graded by number of edges of critical representatives. Finally, we answer various enumerative questions related to $EP_{n}$, adapting methods of Callan and Stein-Everett.
\end{abstract}

\tableofcontents

\section{Introduction}

\numberwithin{equation}{section}

Circular planar electrical networks are a natural generalization of an idea from classical physics: that any electrical resistor network with two vertices connected to a battery behaves like one with a single resistor. When we embed a resistor network in a disk and allow arbitrarily many boundary vertices connected to batteries, the situation becomes more interesting. An inverse boundary problem for these electrical networks was studied in detail by de Verdi\`{e}re-Gitler-Vertigan \cite{french} and Curtis-Ingerman-Morrow \cite{curtis}: given the \textit{response matrix} of a network, that is, information about how the network responds to voltages applied at the boundary vertices, can the network be recovered?

In general, the answer is ``no,'' though much can be said about the information that can be recovered. If, for example, the underlying graph of the electrical network is known and is \textit{critical}, the conductances (equivalently, resistances) can be uniquely recovered \cite[Theorem 2]{curtis}. Moreover, any two networks which produce the same response matrix can be related by a certain class of combinatorial transformations, the \textit{local equivalences} \cite[Th\'{e}or\`{e}me 4]{french}.

The goal of this paper is to study more closely the rich theory of circular planar electrical networks. Specifically, we define a poset $EP_{n}$ of circular planar graphs, under the operations of contraction and deletion of edges, and investigate its properties.

For instance, the poset $EP_{n}$ has an intricate topological structure. By \cite[Theorem 4]{curtis} and \cite[Th\'{e}or\`{e}me 3]{french}, the space of response matrices for circular planar electrical networks of order $n$ decomposes as a disjoint union of open \textit{cells}, each diffeomorphic to a product of copies of the positive real line. In light of this decomposition, we can describe $EP_{n}$ as the poset of these cells under containment of closure:

\newtheorem*{introclosureorder}{Theorem \ref{closureorder}}
\begin{introclosureorder}
$[H]\le [G]$ in $EP_{n}$ if and only if $\Omega(H)\subset\overline{\Omega(G)}$, where $\Omega(H)$ denotes the space of response matrices for conductances on $H$.
\end{introclosureorder}

Using the important tool of \textit{medial graphs} developed in \cite{curtis} and \cite{french}, we also prove:

\newtheorem*{intrograded}{Theorem \ref{graded}}
\begin{intrograded}
$EP_{n}$ is graded by number of edges of critical representatives.
\end{intrograded}

In the last part of the paper, we embark on a study of the enumerative properties of $EP_{n}$. Medial graphs bear a strong resemblance to certain objects whose enumerative properties are known: stabilized-interval free (SIF) permutations, as studied by Callan \cite{callan}, and irreducible linked diagrams, as studied by Stein-Everett \cite{stein}. Exploiting this resemblance, we summarize and prove analogues of known results in the following theorem: 

\newtheorem*{introenumsummary}{Theorem \ref{enumsummary}}
\begin{introenumsummary}
Put $X_{n}=|EP_{n}|$, the number of equivalence classes of electrical networks of order $n$. Then:
\begin{enumerate}
\item[(a)] $X_{1}=1$ and 
\begin{equation*}
X_{n}=2(n-1)X_{n-1}+\sum_{j=2}^{n-2}(j-1)X_{j}X_{n-j}.
\end{equation*}
\item[(b)] $[t^{n-1}]X(t)^{n}=n\cdot(2n-3)!!$, where $X(t)$ is the generating function for the sequence $\{X_{i}\}$.
\item[(c)] $X_{n}/(2n-1)!! \to e^{-1/2}$ as $n\to\infty$.
\end{enumerate}
\end{introenumsummary}

The roadmap of the paper is as follows. In an attempt to keep the exposition as self-contained as possible, we carefully review terminology and known results in \S\ref{defs}, where we also establish some basic properties of electrical networks. In \S\ref{poset}, we define the poset $EP_{n}$, and establish the equivalence of our two characterizations in Theorem \ref{closureorder}. We then prove Theorem \ref{graded}, and investigate various conjectural properties of $EP_{n}$, such as Eulerianness. The study of enumerative properties of $EP_{n}$ is undertaken in \S\ref{enumerative}, where we prove the three parts of Theorem \ref{enumsummary}, and conclude by studying the rank sizes $|EP_{n,r}|$.

\section{Electrical Networks}\label{defs}

\numberwithin{equation}{subsection}

We begin a systematic discussion of electrical networks by recalling various notions and results from \cite{curtis}. We will also introduce some new terminology and conventions which will aid our exposition, in some cases deviating from \cite{curtis}.

\subsection{Circular Planar Electrical Networks, up to equivalence}
 
\begin{defn} A \textbf{circular planar graph} $\Gamma$ is a planar graph embedded in a disk $D$. $\Gamma$ is allowed to have self-loops and multiple edges, and has at least one vertex on the boundary of $D$ - such vertices are called \textbf{boundary vertices}. A \textbf{circular planar electrical network} is a circular planar graph $\Gamma$, together with a \textbf{conductance map} $\gamma:E(\Gamma)\rightarrow\mathbb{R}_{>0}$. \end{defn}

To avoid cumbersome language, we will henceforth refer to these objects as \textbf{electrical networks.} We will also call the number of boundary vertices of an electrical network (or a circular planar graph) its \textbf{order}.

We can interpret this construction as an electrical network in the physical sense, with a resistor existing on each edge $e$ with conductance $\gamma(e)$. Electrical networks satisfy \textbf{Ohm's Law} and \textbf{Kirchhoff's Laws}, classical physical phenomena which we neglect to explain in detail here. Given an electrical network $(\Gamma,\gamma)$, suppose that we apply electrical potentials at each of the boundary vertices $V_{1},\ldots,V_{n}$, inducing currents through the network. Then, we get a map $f:\mathbb{R}^{n}\rightarrow\mathbb{R}^{n}$, where $f$ sends the potentials $(p_{1},\ldots,p_{n})$ applied at the vertices $V_{1},\ldots,V_{n}$ to the currents $(i_{1},\ldots,i_{n})$ observed at $V_{1},\ldots,V_{n}$. We will take currents going out of the boundary to be negative and those going in to the boundary to be positive.

\begin{rmk}
The convention for current direction above is the opposite of that used in \cite{curtis}, but we will prefer it for the ensuing elegance of the statement of Theorem \ref{circularminorspositive}a.
\end{rmk}

In fact, $f$ is linear (see \cite[\S 1]{curtis}), and we have natural bases for the spaces of applied voltages and observed currents at the boundary vertices. Thus, we can make the following definition:

\begin{defn}
Given an electrical network $(\Gamma,\gamma)$, define the \textbf{response matrix} of the network to be the linear map $f$ constructed above from applied voltages to observed currents, in terms of the natural bases indexed by the boundary vertices.
\end{defn}

\begin{defn}
Two electrical networks $(\Gamma_{1},\gamma_{1}),(\Gamma_{2},\gamma_{2})$ are \textbf{equivalent} if they have the same response matrix. In other words, the two networks cannot be distinguished only by applying voltages at the boundary vertices and observing the resulting currents. The equivalence relation is denoted by $\sim$.
\end{defn}

We will study electrical networks up to equivalence. We have an important class of combinatorial transformations that may be applied to electrical networks, known as \textbf{local equivalences}, described below. These transformations may be seen to be equivalences by applications of Ohm's and Kirchhoff's Laws. Note that all of these local equivalences may be performed in reverse.

\begin{enumerate}
\item \textbf{Self-loop and spike removal.} Self-loops (cycles of length 1) and spikes (edges adjoined to non-boundary vertices of degree 1) of any conductances may always be removed.
\item \textbf{Replacement of edges in parallel.} Two edges $e_{1},e_{2}$ between with common endpoints $v,w$ may be replaced by a single edge of conductance $\gamma(e_{1})+\gamma(e_{2})$.
\item \textbf{Replacement of edges in series.} Two edges $v_{1}w,wv_{2}$ ($v_{1}\neq v_{2}$) meeting at a vertex $w$ of degree 2 may be replaced by a single edge $v_{1}v_{2}$ of conductance $((\gamma(v_{1}w)^{-1}+\gamma(wv_{2})^{-1})^{-1}$.
\item \textbf{Y-$\Delta$ transformations.} (See Figure \ref{YD}) Three edges $v_{1}w,v_{2}w,v_{3}w$ meeting at a non-boundary vertex $w$ of degree 3 may be replaced by three edges $v_{1}v_{2},v_{2}v_{3},v_{3}v_{1}$, of conductances 
\begin{equation*}
\frac{\gamma_{1}\gamma_{2}}{\gamma_{1}+\gamma_{2}+\gamma_{3}},\frac{\gamma_{2}\gamma_{3}}{\gamma_{1}+\gamma_{2}+\gamma_{3}},\frac{\gamma_{3}\gamma_{1}}{\gamma_{1}+\gamma_{2}+\gamma_{3}},
\end{equation*}
where $\gamma_{i}$ denotes the conductance $\gamma(v_{i}w)$.
\end{enumerate}

\begin{figure}
\begin{center}
\includegraphics[scale=0.4]{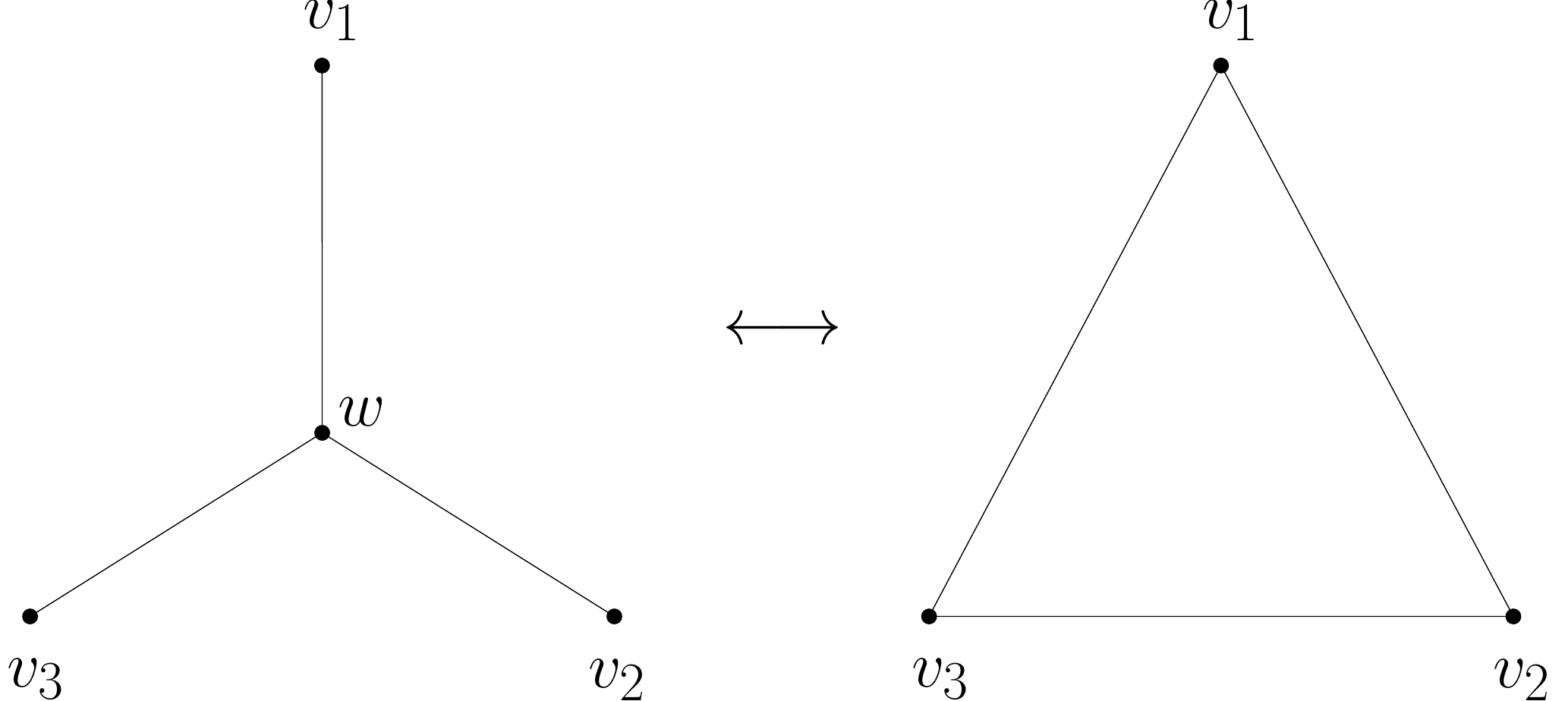}
\caption{Y-$\Delta$ transformation.}\label{YD}
\end{center}
\end{figure}

In fact, local equivalences are sufficient to generate equivalence of electrical networks:

\begin{theorem}[{\cite[Th\'{e}or\`{e}me 4]{french}}]\label{localeqenough}
Two electrical networks are equivalent if and only if they are related by a sequence of local equivalences.
\end{theorem}

When dealing with electrical networks, we will sometimes avoid making any reference to the conductance map $\gamma$, and instead consider just the underlying circular planar graph $\Gamma$. In doing so, we will abuse terminology by calling circular planar graphs ``electrical networks.'' In practice, we will only use the following notion of equivalence on circular planar graphs:

\begin{defn}
Let $\Gamma_{1},\Gamma_{2}$ be circular planar graphs, each with the same number of boundary vertices. Then, $\Gamma_{1},\Gamma_{2}$ are \textbf{equivalent} if there exist conductances $\gamma_{1},\gamma_{2}$ on $\Gamma_{1},\Gamma_{2}$, respectively such that $(\Gamma_{1},\gamma_{1}),(\Gamma_{2},\gamma_{2})$ are equivalent electrical networks. As with electrical networks, this equivalence is denoted $\sim$.
\end{defn}

It is clear that we may still apply local equivalences with this notion of equivalence. Furthermore, we have an analogue of Theorem \ref{localeqenough}: two circular planar graphs are equivalent if and only if they can be related by local equivalences, where we ``forget'' the conductances.

\subsection{Circular Pairs and Circular Minors}

Circular pairs and circular minors are central to the characterization of response matrices.

\begin{defn}\label{circpair}
Let $P=\{p_{1},p_{2},\ldots,p_{k}\}$ and $Q=\{q_{1},q_{2},\ldots,q_{k}\}$ be disjoint ordered subsets of the boundary vertices of an electrical network $(\Gamma,\gamma)$. We say that $(P;Q)$ is a \textbf{circular pair} if $p_{1},\ldots,p_{k},q_{k},\ldots,q_{1}$ are in clockwise order around the circle. We will refer to $k$ as the \textbf{size} of the circular pair.
\end{defn}

\begin{rmk}\label{flipPQ}
We will take $(P;Q)$ to be the same circular pair as $(\widetilde{Q};\widetilde{P})$, where $\widetilde{P}$ denotes the ordered set $P$ with its elements reversed. Almost all of our definitions and statements are compatible with this convention; most notably, by Theorem \ref{circularminorspositive}a, because response matrices are positive, the circular minors $M(P;Q)$ and $M(\widetilde{Q};\widetilde{P})$ are the same. Whenever there is a question as to the effect of choosing either $(P;Q)$ or $(\widetilde{Q};\widetilde{P})$, we take extra care to point the possible ambiguity.
\end{rmk}

\begin{defn}\label{connection}
Let $(P;Q)$ and $(\Gamma,\gamma)$ be as in Definition \ref{circpair}. We say that there is a \textbf{connection} from $P$ to $Q$ in $\Gamma$ if there exists a collection of vertex-disjoint paths from $p_{i}$ to $q_{i}$ in $\Gamma$, and furthermore each path in the collection contains no boundary vertices other than its endpoints. We denote the set of circular pairs $(P;Q)$ for which $P$ is connected to $Q$ by $\pi(\Gamma)$.
\end{defn}

\begin{defn}
Let $(P;Q)$ and $(\Gamma,\gamma)$ be as in Definition \ref{circpair}, and let $M$ be the response matrix. We define the \textbf{circular minor} associated to $(P;Q)$ to be the determinant of the $k\times k$ matrix $M(P;Q)$ with $M(P;Q)_{i,j}=M_{p_{i},q_{j}}$.
\end{defn}

\begin{rmk}
We will sometimes refer to submatrices and their determinants both as minors, interchangeably. In all instances, it will be clear from context which we mean.
\end{rmk}

We are interested in circular minors and connections because of the following result from \cite{curtis}:

\begin{theorem}\label{circularminorspositive}
Let $M$ be an $n\times n$ matrix. Then:
\begin{enumerate}
\item[(a)] $M$ is the response matrix for an electrical network $(\Gamma,\gamma)$ if and only if $M$ is symmetric with row and column sums equal zero, and each of the circular minors $M(P;Q)$ is non-negative. 
\item[(b)] If $M$ is the response matrix for an electrical network $(\Gamma,\gamma)$, the positive circular minors $M(P;Q)$ are exactly those for which there is a connection from $P$ to $Q$.
\end{enumerate}
\end{theorem}

\begin{proof}
(a) is immediate from \cite[Theorem 4]{curtis}, which we will state as Theorem \ref{diffeo} later. (b) is \cite[Theorem 4.2]{curtis}. Note that, because we have declared current going into the circle to be negative, we do not have the extra factors of $(-1)^{k}$ as in \cite{curtis}.
\end{proof}

We now define two operations on the circular planar graphs and electrical networks. Each operation decreases the total number of edges by one.

\begin{defn}
Let $G$ be a circular planar graph, and let $e$ be an edge with endpoints $v,w$. The \textbf{deletion} of $e$ from $G$ is exactly as named; the edge $e$ is removed while leaving the rest of the vertices and edges of $G$ unchanged. If $v,w$ are not both boundary vertex of $G$, we may also perform a \textbf{contraction} of $e$, which identifies all points of $e$. If exactly one of $v,w$ is a boundary vertex, then the image of $e$ under the contraction is a boundary vertex. Note that edges connecting two boundary vertices cannot be contracted to either endpoint. See Figure \ref{deletecontract}.
\end{defn}

\begin{figure}
        \centering
        \begin{subfigure}[t]{0.3\textwidth}
                \centering
                \includegraphics[width=\textwidth]{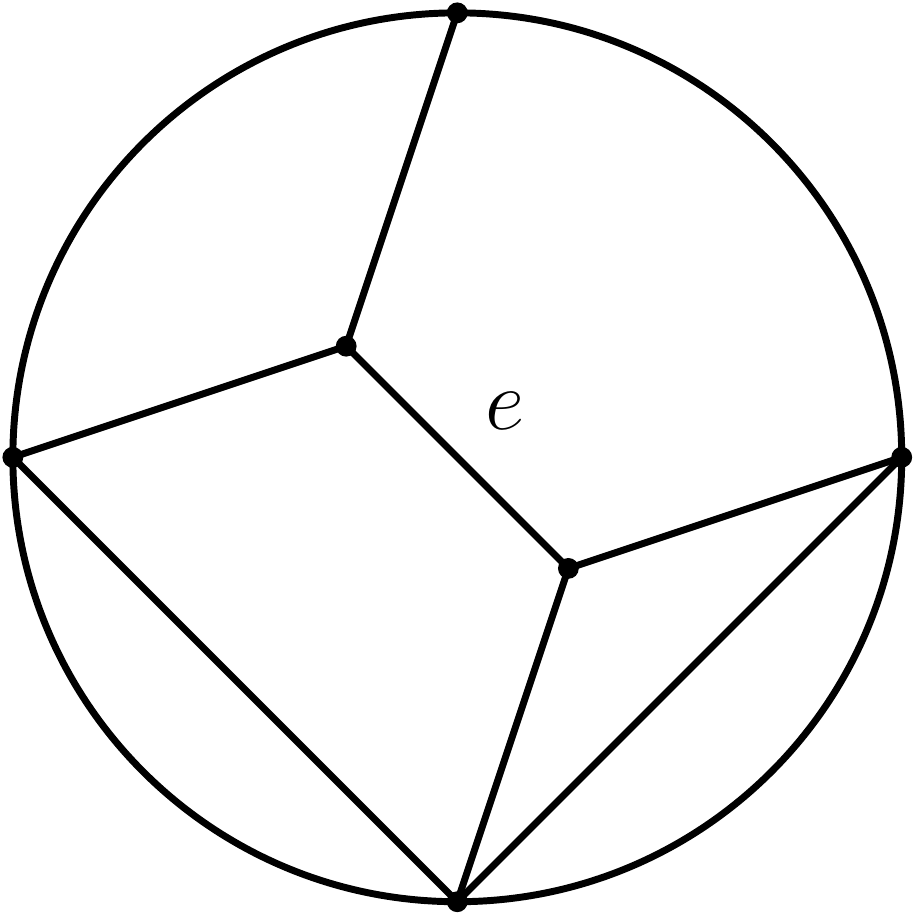}
                \caption{A circular planar graph $G$}
        \end{subfigure}%
        ~ %add desired spacing between images, e. g. ~, \quad, \qquad etc.
        \begin{subfigure}[t]{0.3\textwidth}
                \centering
                \includegraphics[width=\textwidth]{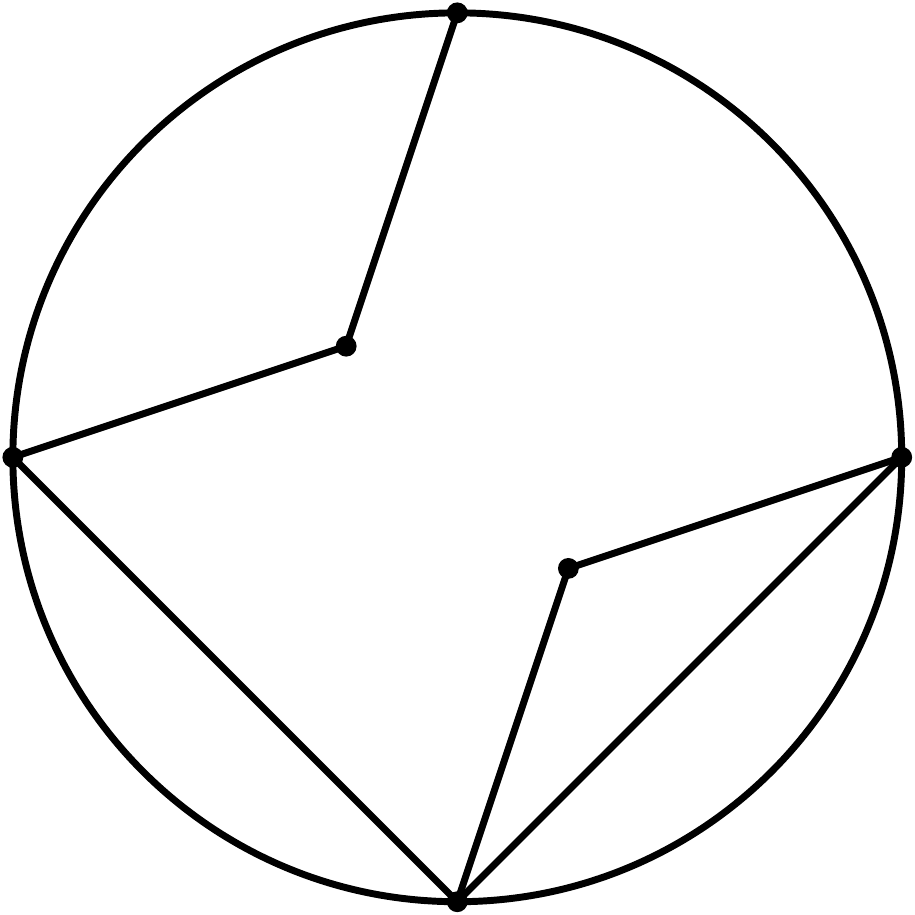}
                \caption{The deletion of edge $e$ from $G$}
        \end{subfigure}
        ~ %add desired spacing between images, e. g. ~, \quad, \qquad etc.
        \begin{subfigure}[t]{0.3\textwidth}
                \centering
                \includegraphics[width=\textwidth]{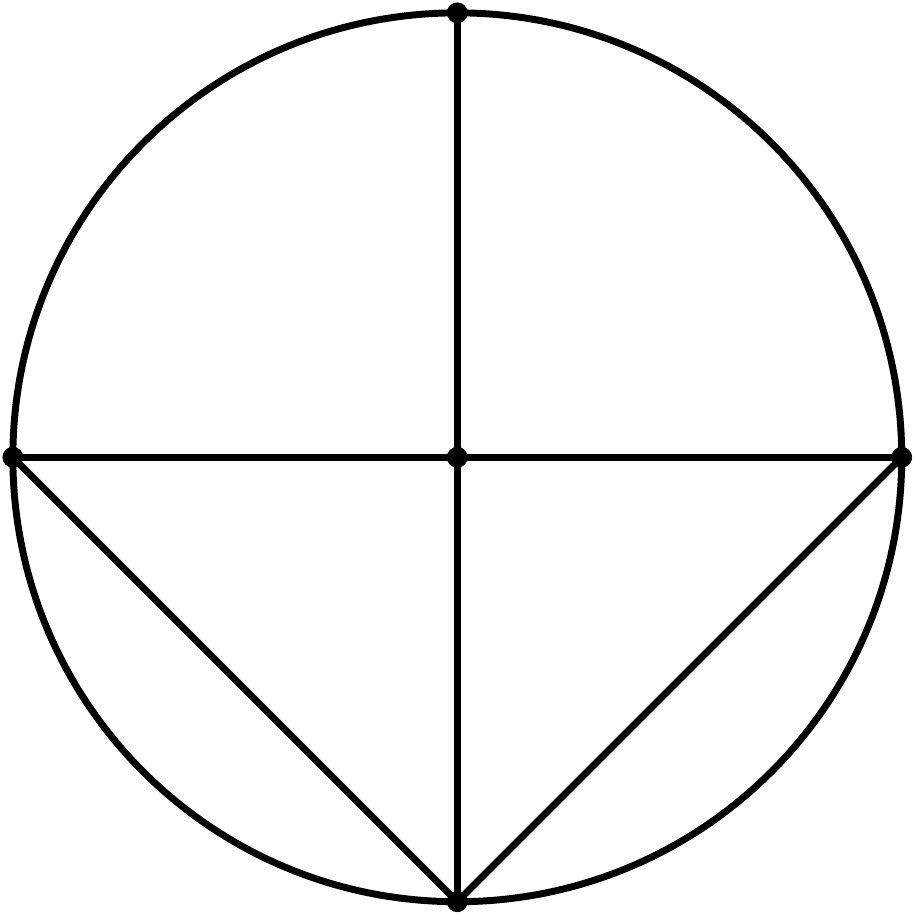}
                \caption{The contraction of edge $e$ from $G$}
        \end{subfigure}
        \caption{An example of a deletion and a contraction}\label{deletecontract}
\end{figure}

\subsection{Critical Graphs}

In this section, we introduce critical graphs, a particular class of circular planar graphs.

\begin{defn}
Let $G$ be a circular planar graph. $G$ is said to be \textbf{critical} if, for any removal of an edge via deletion or contraction, there exists a circular pair $(P;Q)$ for which $P$ is connected to $Q$ through $G$ before the edge removal, but not afterward.
\end{defn}

\begin{theorem}[{\cite[Th\'{e}or\`{e}me 2]{french}}]\label{critrep}
Every equivalence class of circular planar graphs has a critical representative.
\end{theorem}

\begin{theorem}[{\cite[Theorem 1]{curtis}}]\label{equivalentconnections}
Suppose $G_{1},G_{2}$ are critical. Then, $G_{1}$ and $G_{2}$ are Y-$\Delta$ equivalent (that is, related by a sequence of Y-$\Delta$ transformations) if and only if $\pi(G_{1})=\pi(G_{2})$.
\end{theorem}

\begin{proposition}\label{equalconnectionset}
Let $G_{1},G_{2}$ be arbitrary circular planar graphs. Then, $G_{1}\sim G_{2}$ if and only if $\pi(G_{1})=\pi(G_{2})$.
\end{proposition}

\begin{proof}
By Theorem \ref{localeqenough}, if $G_{1}\sim G_{2}$, then $G_{1}$ and $G_{2}$ are related by a sequence of local equivalences. All local equivalences preserve $\pi(-)$; indeed, we have the claim for Y-$\Delta$ transformations by Theorem \ref{equivalentconnections}, and it is easy to check for all other local equivalences. It follows that $\pi(G_{1})=\pi(G_{2})$. 

In the other direction, by Theorem \ref{critrep}, there exist critical graphs $H_{1},H_{2}$ such that $G_{1}\sim H_{1}$ and $G_{2}\sim H_{2}$. By similar logic from the previous paragraph, we have $\pi(H_{1})=\pi(G_{1})=\pi(G_{2})=\pi(H_{2})$, and thus, by Theorem \ref{equivalentconnections}, $H_{1}\sim H_{2}$. It follows that $G_{1}\sim G_{2}$, so we are done.
\end{proof}

\begin{defn}
Fix a set $B$ of $n$ boundary vertices on a disk $D$. For any set of circular minors $\pi$, let $\Omega(\pi)$ denote the set of response matrices whose, with the set of positive minors being exactly those corresponding to the elements $\pi$. We will refer to the sets $\Omega(\pi)$ as \textbf{cells}, in light of the theorem that follows.
\end{defn}

\begin{theorem}[{\cite[Theorem 4]{curtis}}]\label{diffeo}
Suppose that $G$ is critical and has $N$ edges. Put $\pi=\pi(G)$. Then, the map $r_{G}:\mathbb{R}_{>0}^{N}\rightarrow\Omega(\pi)$, taking the conductances on the edges of $G$ to the resulting response matrix, is a diffeomorphism.
\end{theorem}

It follows that the space of response matrices for electrical networks of order $n$ is the disjoint union of the cells $\Omega(\pi)$, some of which are empty. The non-empty cells $\Omega(\pi)$ are those which correspond to critical graphs $G$ with $\pi(G)=\pi$. We will describe how these cells are attached to each other in Proposition \ref{closure}.

\begin{rmk}\label{omegapi}
Later, we will prefer (e.g. in Proposition \ref{closure}) to index these cells by their underlying (equivalence classes of) circular planar graphs, referring to them as $\Omega(G)$. Thus, $\Omega(G)$ denotes the set response matrices for conductances on $G$.
\end{rmk}

Let us now characterize critical graphs in more conceptual ways. We quote a fourth characterization using medial graphs in Theorem \ref{criticalmedial}.

\begin{theorem}\label{criticalcharacterization}
Let $(\Gamma,\gamma)$ be an electrical network. The following are equivalent:
\begin{enumerate}
\item[(1)] $\Gamma$ is critical.
\item[(2)] Given the response matrix $M$ of $(\Gamma,\gamma)$, $\gamma$ can uniquely be recovered from $M$ and $\Gamma$.
\item[(3)] $\Gamma$ has the minimal number of edges among elements of its equivalence class.
\end{enumerate}
\end{theorem}

\begin{proof}
By \cite[Lemma 13.2]{curtis}, (1) and (2) are equivalent. We now prove that (3) implies (1). Suppose for sake of contradiction $\Gamma$ has the minimal number of edges among elements of its equivalence class, but is not critical. Then, there exists some edge that may be contracted or deleted to give $\Gamma'$, such that $\pi(\Gamma)=\pi(\Gamma')$. Then, by Proposition \ref{equalconnectionset}, we have $\Gamma'\sim\Gamma$, contradicting the minimality of the number of edges of $\Gamma$.

Finally, to see that (1) implies (3), suppose for sake of contradiction that $\Gamma$ is critical and equivalent to a graph $\Gamma'$ with a strictly fewer edges. $\Gamma'$ cannot be critical, or else $\Gamma$ and $\Gamma'$ would be Y-$\Delta$ equivalent and thus have an equal number of edges. However, if $\Gamma'$ is not critical, we also have a contradiction by the previous paragraph. The result follows.
\end{proof}

\subsection{Medial Graphs}

One of our main tools in studying circular planar graphs (and thus, electrical networks) will be their medial graphs. In a sense, medial graphs are the dual object to circular planar graphs. See Figure \ref{medialgeodesics} for examples. 

Let $G$ be a circular planar graph with $n$ boundary vertices; color all vertices of $G$ black, for convenience. Then, for each boundary vertex, add two red vertices to the boundary circle, one on either side, as well as a red vertex on each edge of $G$. We then construct the \textbf{medial graph} of $G$, denoted $\mc{M}(G)$, as follows. 

Take the set of red vertices to be the vertex set of $\mc{M}(G)$. Two red non-boundary vertices in $\mc{M}(G)$ are connected by an edge if and only if their associated edges share a vertex and border the same face. Then, the red boundary vertices are each connected to exactly one other red vertex: if the red boundary vertex $r$ lies clockwise from its associated black boundary vertex $b$, then $r$ is connected to the red vertex associated to the first edge in clockwise order around $b$ after the arc $rb$. Similarly, if $r$ lies counterclockwise from $b$, we connect $r$ to the red vertex associated to the first edge in counterclockwise order around $b$ after the arc $rb$. Note that if no edges of $G$ are incident at $b$, then the two red vertices associated to $b$ are connected by an edge of $\mc{M}(G)$.

\begin{figure}
\begin{center}
\includegraphics[scale=0.4]{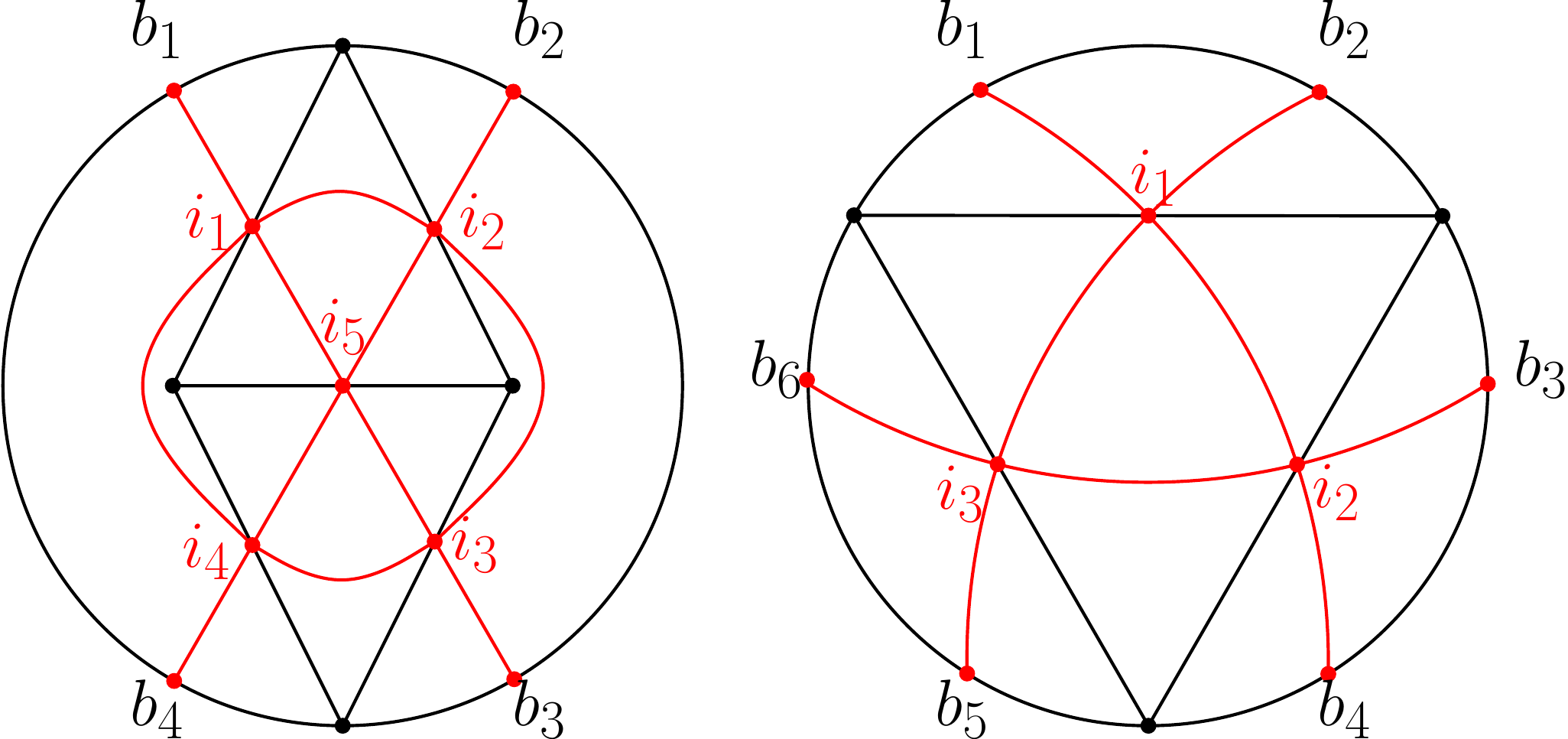}
\end{center}
\caption{Medial Graphs}\label{medialgeodesics}
\end{figure}

We will refer to the red vertices on the boundary circle as \textbf{medial boundary vertices}, as to distinguish them from the black \textbf{boundary vertices}, a term we will reserve for the boundary vertices of the original circular planar graph (electrical network). The \textbf{order} of a medial graph is the order of its underlying electrical network.

Note that the medial boundary vertices of $\mc{M}(G)$ have degree 1, and all other vertices have degree 4. Thus, we may form \textbf{geodesics} in $\mc{M}(G)$ in the following way. Starting at each medial boundary vertex, draw a path $e_{1}e_{2}\cdots e_{n}$ (labeled by its edges) so that if the edge $e_{i}$ ends at the non-medial boundary vertex $v$, the edge $e_{i+1}$ is taken to be the edge with endpoint $v$ such that the edges $e_{i}$ and $e_{i+1}$ separate the other two edges incident at $v$. The geodesic ends when it reaches a second boundary vertex. The remaining geodesics are constructed in a similar way, but do not start and end at boundary vertices: instead, they must be finite cycles inside the circle.

For example, in Figure \ref{medialgeodesics}, we have three geodesics in the right hand diagram: $b_{1}b_{4}$, $b_{2}b_{5}$, and $b_{3}b_{6}$, where here we label the geodesics by their vertices. In the left hand diagram, we have the geodesics $b_{1}b_{3}$, $b_{2}b_{4}$, and $i_{1}i_{2}i_{3}i_{4}$.

\begin{defn}Two geodesics are said to form a \textbf{lens} if they intersect at distinct $p_{1}$ and $p_{2}$, in such a way that they do not intersect between $p_{1}$ and $p_{2}$. A medial graph is said to be \textbf{lensless} if all geodesics connect two medial boundary vertices (that is, no geodesics are cycles), and no two geodesics form a lens.
\end{defn}

The local equivalences of electrical networks may easily be translated into operations on their medial graphs. Most importantly, Y-$\Delta$ transformations become \textbf{motions}, as shown in \ref{motion}, and replacing series or parallel edges with a single edge both correspond to \textbf{resolution} of lenses, as shown in Figure \ref{lensresolution}. Note, however, that a lens may only be resolved if no other geodesics pass through the lens. Defining two medial graphs to be \textbf{equivalent} if their underlying circular planar graphs are equivalent, we obtain an analogue of Theorem \ref{localeqenough}.

\begin{figure}
\begin{center}
\includegraphics[scale=0.3]{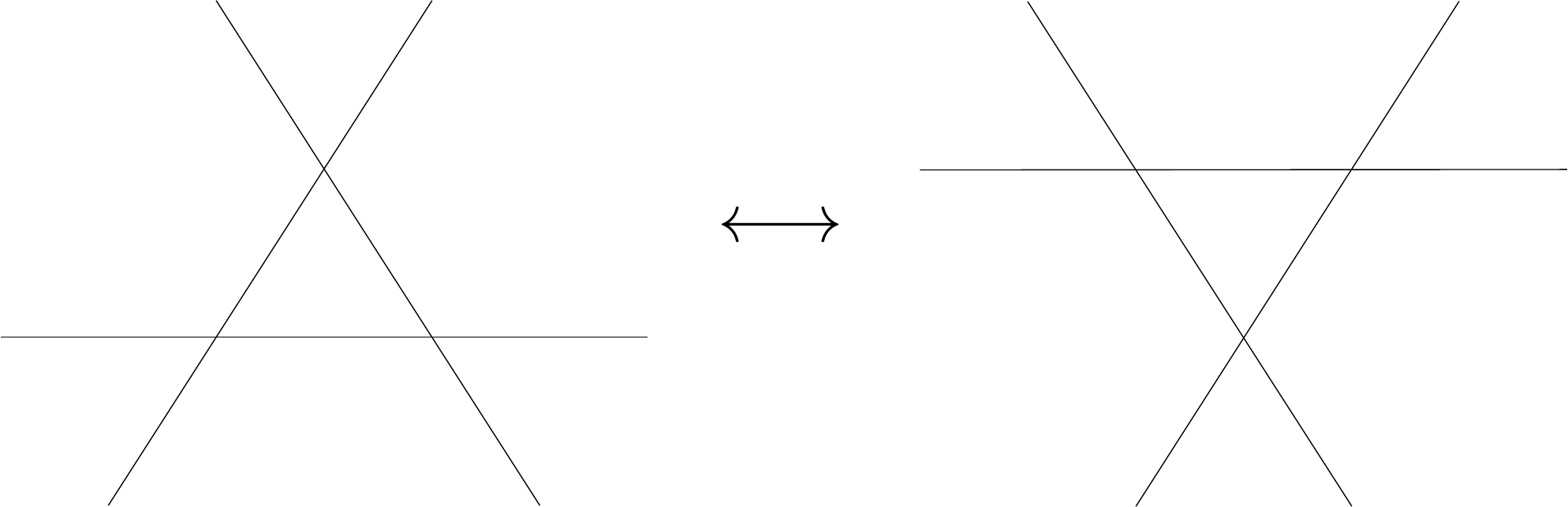}
\caption{Motions}\label{motion}
\end{center}
\end{figure}

\begin{figure}
\begin{center}
\includegraphics[scale=0.6]{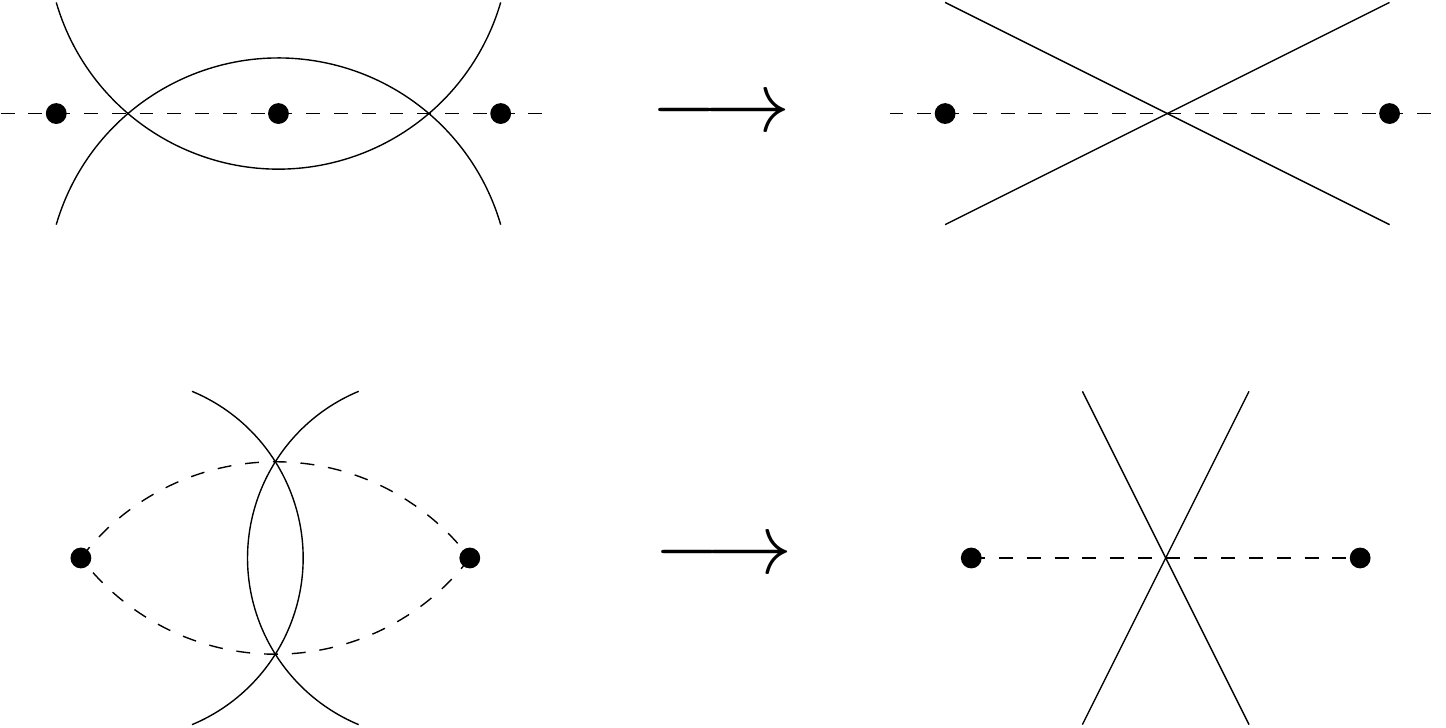}
\caption{Resolution of Lenses}\label{lensresolution}
\end{center}
\end{figure}

The power of medial graphs lies in the following theorem:

\begin{theorem}[{\cite[Lemma 13.1]{curtis}}]\label{criticalmedial}
$G$ is critical if and only if $\mc{M}(G)$ is lensless.
\end{theorem}

In particular, if $G$ is critical, the geodesics of $\mc{M}(G)$ consist only $n$ ``wires'' connecting pairs of the $2n$ boundary medial vertices. Thus, any critical graph $G$ gives a perfect matching of the medial boundary vertices. Furthermore, suppose $H\sim G$ is critical. By Theorem \ref{equivalentconnections} and Proposition \ref{equalconnectionset}, $G$ and $H$ are related by Y-$\Delta$ transformations, so $\mc{M}(G)$ and $\mc{M}(H)$ are related by motions. In particular, $\mc{M}(G)$ and $\mc{M}(H)$ match the same pairs of boundary medial vertices, so we have a well-defined map from critical circular planar graph equivalence classes to matchings. In fact, this map is injective:

\begin{proposition}\label{motionsenough}
Suppose that the geodesics of two lensless medial graphs $\mc{M}(G),\mc{M}(H)$ match the same pairs of medial boundary vertices. Then, the medial $\mc{M}(G)$ and $\mc{M}(H)$ are related by motions, or equivalently, $G$ and $H$ are Y-$\Delta$ equivalent.
\end{proposition}

\begin{proof}
Implicit in \cite[Theorem 7.2]{curtis}.
\end{proof}

\begin{defn}
Given the boundary vertices of a circular planar graph embedded in a disk $D$, take $2n$ medial boundary vertices as before. A \textbf{wiring diagram} is collection of $n$ smooth curves (wires) embedded in $D$, each of which connects a pair of medial boundary vertices in such a way that each medial boundary vertex has exactly one incident wire. We require that wiring diagrams have no triple crossings or self-loops. As with electrical networks and medial graphs, the \textbf{order} of the wiring diagram is defined to be equal to $n$.
\end{defn}

It is immediate from Proposition \ref{motionsenough} that, given a set of boundary vertices, perfect matchings on the set of medial boundary vertices are in bijection with motion-equivalence classes of lensless wiring diagrams. Thus, we have an injection $G\mapsto\mc{M}(G)$ from critical graph equivalence classes to motion-equivalence classes of lensless wiring diagram, but this map is not surjective. We describe the image of this injection in the next definition:

\begin{defn}
Given boundary vertices $V_{1},\ldots,V_{n}$ and a wiring diagram $W$ on the same boundary circle, a \textbf{dividing line} for $W$ is a line $V_{i}V_{j}$ with $i\neq j$ such that there does not exist a wire connecting two points on opposite sides of $V_{i}V_{j}$. The wiring diagram is called \textbf{full} if it has no dividing lines.
\end{defn}

It is obvious that fullness is preserved under motions. Now, suppose that we have a lensless full wiring diagram $W$; we now define a critical graph $\mc{E}(W)$. Let $D$ be the disk in which our wiring diagram is embedded. The wires of $W$ divide $D$ in to faces, and it is well-known that these faces can be colored black and white such that neighboring faces have opposite colors.

The condition that $W$ be full means that each face contains at most one boundary vertex. Furthermore, all boundary vertices are contained in faces of the same color; without loss of generality, assume that this color is black. Then place an additional vertex inside each black face which does not contain a boundary vertex. The boundary vertices, in addition to these added interior vertices, form the vertex set for $\mc{E}(W)$. Finally, two vertices of $\mc{E}(W)$ are connected by an edge if and only if their corresponding faces share a common point on their respective boundaries, which must be an intersection $p$ of two wires of $W$. This edge is drawn as to pass through $p$. An example is shown in Figure \ref{wiringtoelectrical}.

\begin{figure}
\begin{center}
\includegraphics[scale=0.7]{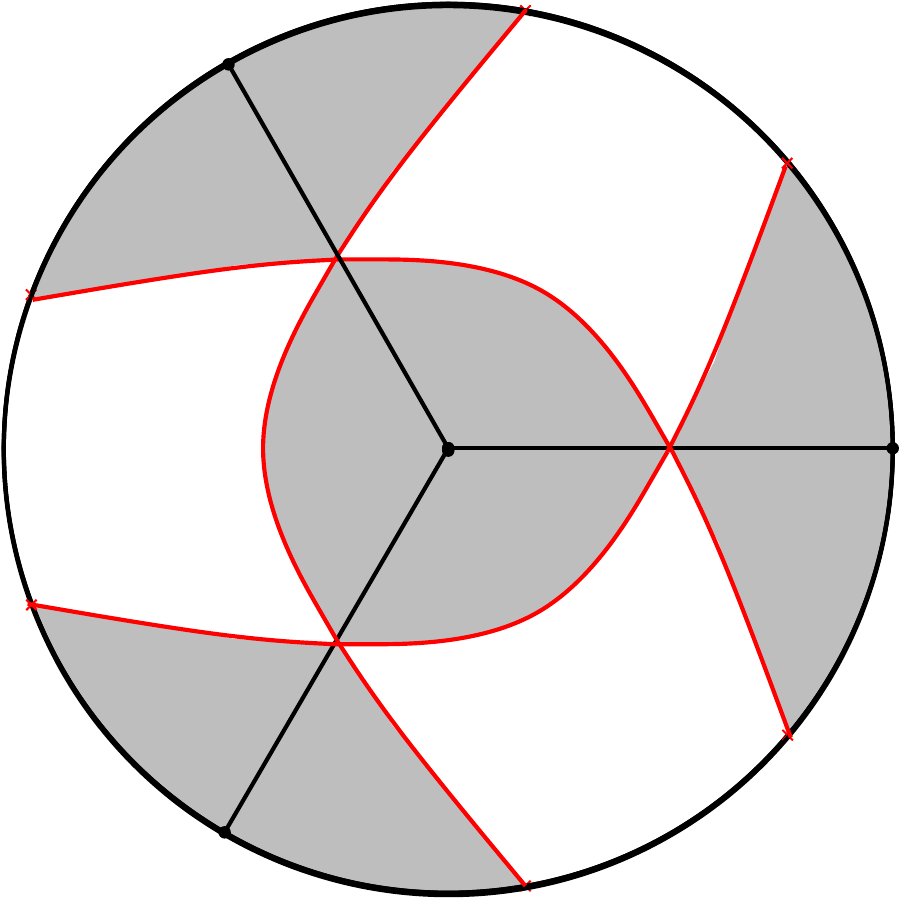}
\caption{Recovering an electrical network from its (lensless) medial graph.}\label{wiringtoelectrical}
\end{center}
\end{figure}

It is straightforward to check that $\mc{M}$ and $\mc{E}$ are inverse maps. We have thus proven the following result:

\begin{theorem}\label{critmedbij}
The associations $G\mapsto\mc{M}(G)$ and $W\mapsto\mc{E}(W)$ are inverse bijections between equivalence classes of critical graphs and motion-equivalence classes of full lensless wiring diagrams.
\end{theorem}

Finally, let us discuss the analogues of contraction and deletion in medial graphs. Each operation corresponds to the \textbf{breaking} of a crossing, as shown in Figure \ref{breaking}. A crossing may be broken in two ways: breaking outward from the corresponding edge of the underlying electrical network corresponds to contraction, and breaking along the edge corresponds to deletion. In the same way that contraction or deletion of an edge in a critical graph is not guaranteed to yield a critical graph, breaking a crossing in lensless medial graphs does not necessarily yield a lensless medial graph.

\begin{figure}
\begin{center}
\includegraphics[scale=0.4]{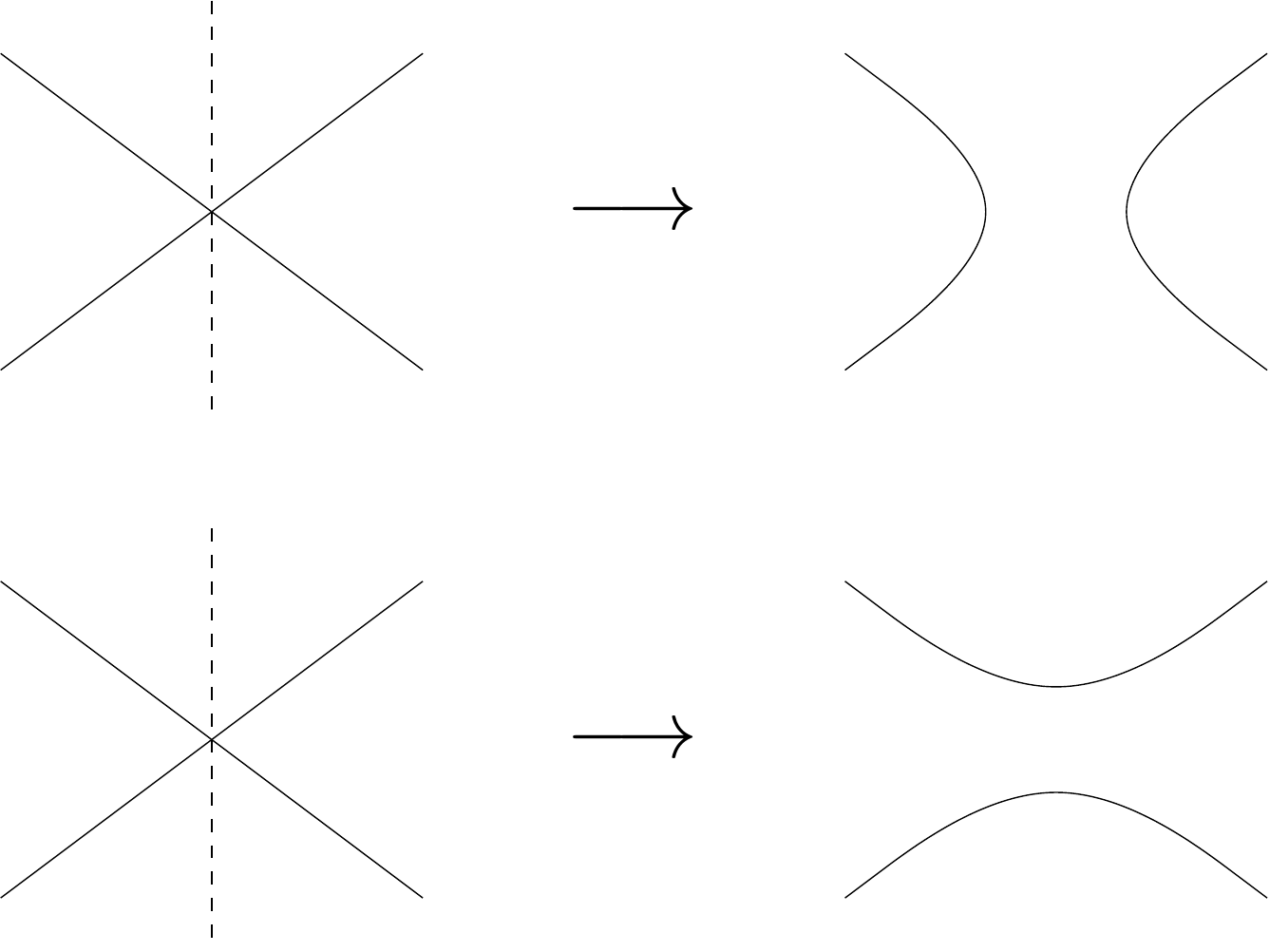}
\caption{Breaking a crossing, in two ways.}\label{breaking}
\end{center}
\end{figure}

Not all breakings of crossings are valid, as some crossings may be broken in a particular way to create a dividing line. In fact, it is straightforward to check that creating a dividing line by breaking a crossing corresponds to contracting a boundary edge, which we also do not allow. Thus, we allow all breakings of crossings as long as no dividing lines are created; such breakings are called \textbf{legal}.

\section{The Electrical Poset $EP_{n}$}\label{poset}

We now consider $EP_{n}$, the poset of circular planar graphs under contraction and deletion. We will find that, equivalently, $EP_{n}$ is the poset of disjoint cells $\Omega(G)$ (see Remark \ref{omegapi}) under containment in closure.

\subsection{Construction}

Before constructing $EP_{n}$, we need a lemma to guarantee that the order relation will be well-defined.

\begin{lemma}\label{orderwelldefined}
Let $G$ be a circular planar graph, and suppose that $H$ can be obtained from $G$ by a sequence of contractions and deletions. Consider a circular planar graph $G'$ with $G'\sim G$. Then, there exists a sequence of contractions and deletions starting from $G'$ whose result is some $H'\sim H$.
\end{lemma}
\begin{proof}
By induction, we may assume that $H$ can be obtained from $G$ by one contraction or one deletion. Furthermore, by Theorem \ref{localeqenough}, we may assume by induction that $G$ and $G'$ are related by a local equivalence. If this local equivalence is the deletion of a self-loop or boundary spike, the result is trivial. Next, suppose $G'$ is be obtained from $G$ by one Y-$\Delta$ transformation. We have several cases: in each, let the vertices of the Y (and $\Delta$) to which the transformation is applied be $A,B,C$, and let the central vertex of the Y, which may be in $G$ or $G'$ be $P$. In each case, if the deleted or contracted edge of $G$ is outside the Y or $\Delta$, it is clear that the same edge-removal may be performed in $G'$.
\begin{itemize}
\item 
Suppose that a Y in $G$ may be transformed to a $\Delta$ in $G'$, and that $H$ is obtained from $G$ by contraction, without loss of generality, of $AP$. Then, deleting the edge $BC$ from $G'$ yields $H'\sim H$.

\item 
Suppose that a Y in $G$ may be transformed to a $\Delta$ in $G'$, and that $H$ is obtained from $G$ by deletion, without loss of generality, of $AP$. Then, deleting $AB$ and $AC$ from $G'$ yields $H'\sim H$.

\item
Suppose that a $\Delta$ in $G$ may be transformed to a Y in $G'$, and that $H$ is obtained from $G$ by deletion, without loss of generality, of $AB$. Then, contracting $CP$ in $G'$ yields $H'\sim H$.

\item
Suppose that a $\Delta$ in $G$ may be transformed to a Y in $G'$, and that $H$ is obtained from $G$ by deletion, without loss of generality, of $AB$. Then, contracting $AP$ to $A$ and $AB$ to $B$ in $G'$ yields $H'\sim H$.

\end{itemize}

Next, consider the case in which we have parallel edges $e,f$ connecting the vertices $A,B$ in $G$, and that $G'$ is obtained by removing $e$ (analogous to replacing the parallel edges by a single edge). If, in $G$, we contract or delete an edge not connecting $A$ and $B$ to get $H$, we can perform the same operation in $G'$ and then delete $E$ to get $H'\sim H$. If, instead, we contract an edge between $A$ and $B$ to get $H$ from $G$, we perform the same operation in $G'$, and then delete $e$, which became a self-loop. Finally, if we delete an edge between $A$ and $B$ to get $H$, then we can delete the same edge in $G'$ to get $H$, unless $e$ is deleted from $G$, in which case we can take $H'=H$.

Now, suppose $G'$ can be obtained from $G$ by adding an edge $e$ in parallel to an edge already in $G$. Then, if we contract or delete an edge $f$ in $G$ to get $H$, we can perform the same operation in $G'$, then delete $e$, to get $H'\sim H$. 

The case in which $G'$ and $G$ are related by contracting an edge in series with another edge follows from a similar argument. We have exhausted all local equivalences, completing the proof.
\end{proof}

For distinct equivalence classes $[G],[H]$, we may now define $[H]<[G]$ if, given any $G\in[G]$, there exists a sequence of contractions and deletions that may be applied to $G$ to obtain an element of $[H]$. We thus have a (well-defined) \textbf{electrical poset of order $n$}, denoted $EP_{n}$, of equivalence classes of circular planar graphs or order $n$. If $H\in [H]$ and $G\in [G]$ with $[H]<[G]$, we will write $H<G$.

Figure \ref{EP3} shows $EP_{3}$, with elements represented as medial graphs (left) and electrical networks (right). Theorem \ref{critrep} guarantees that the electrical networks may be taken to be critical. Note that $EP_{3}$ is isomorphic to the Boolean Lattice $B_{3}$, because all critical graphs of order 3 arise from taking edge-subsets of the top graph.

\begin{figure}
\begin{center}
\includegraphics[scale=0.3]{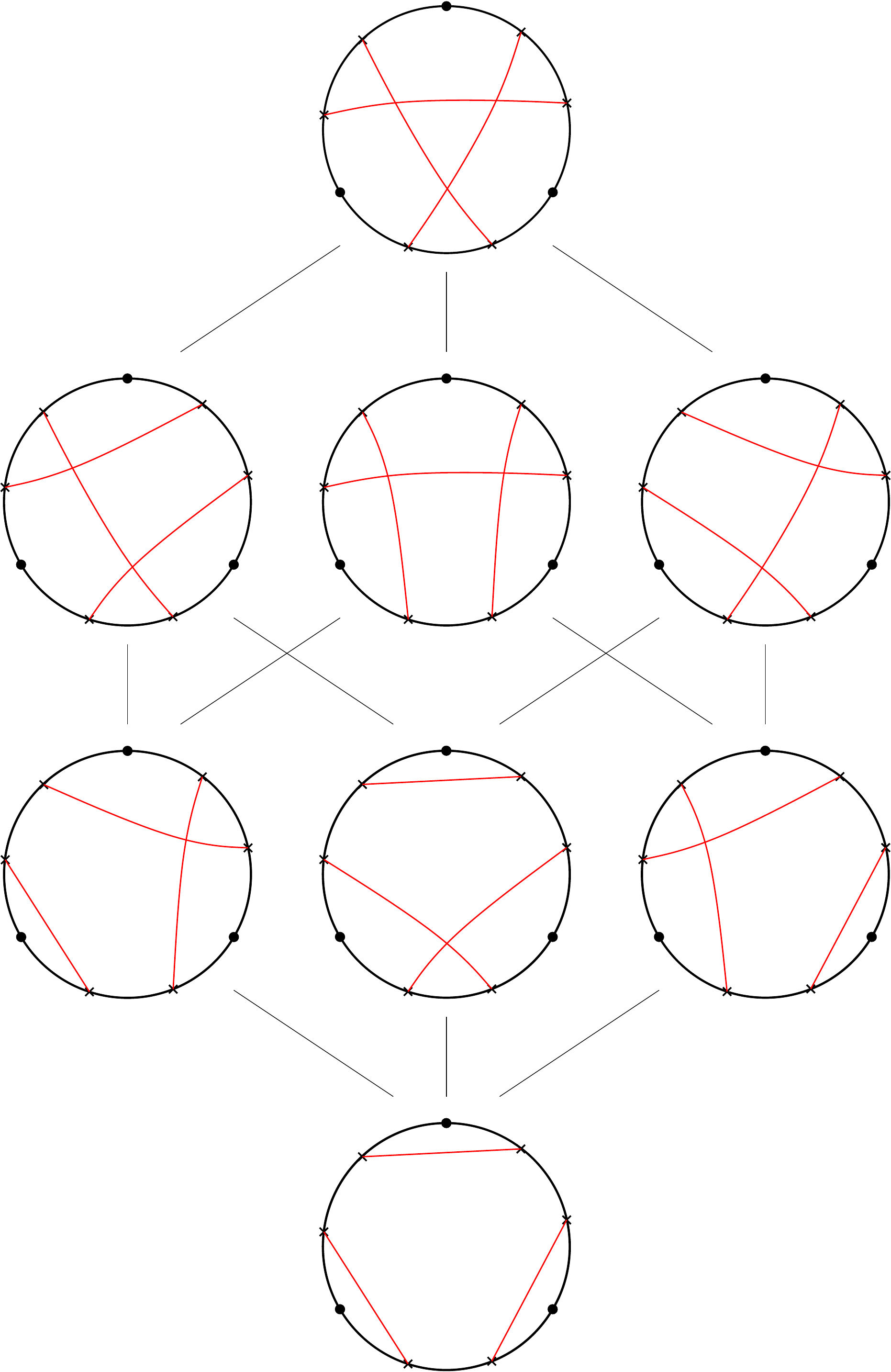}
\qquad
\includegraphics[scale=0.3]{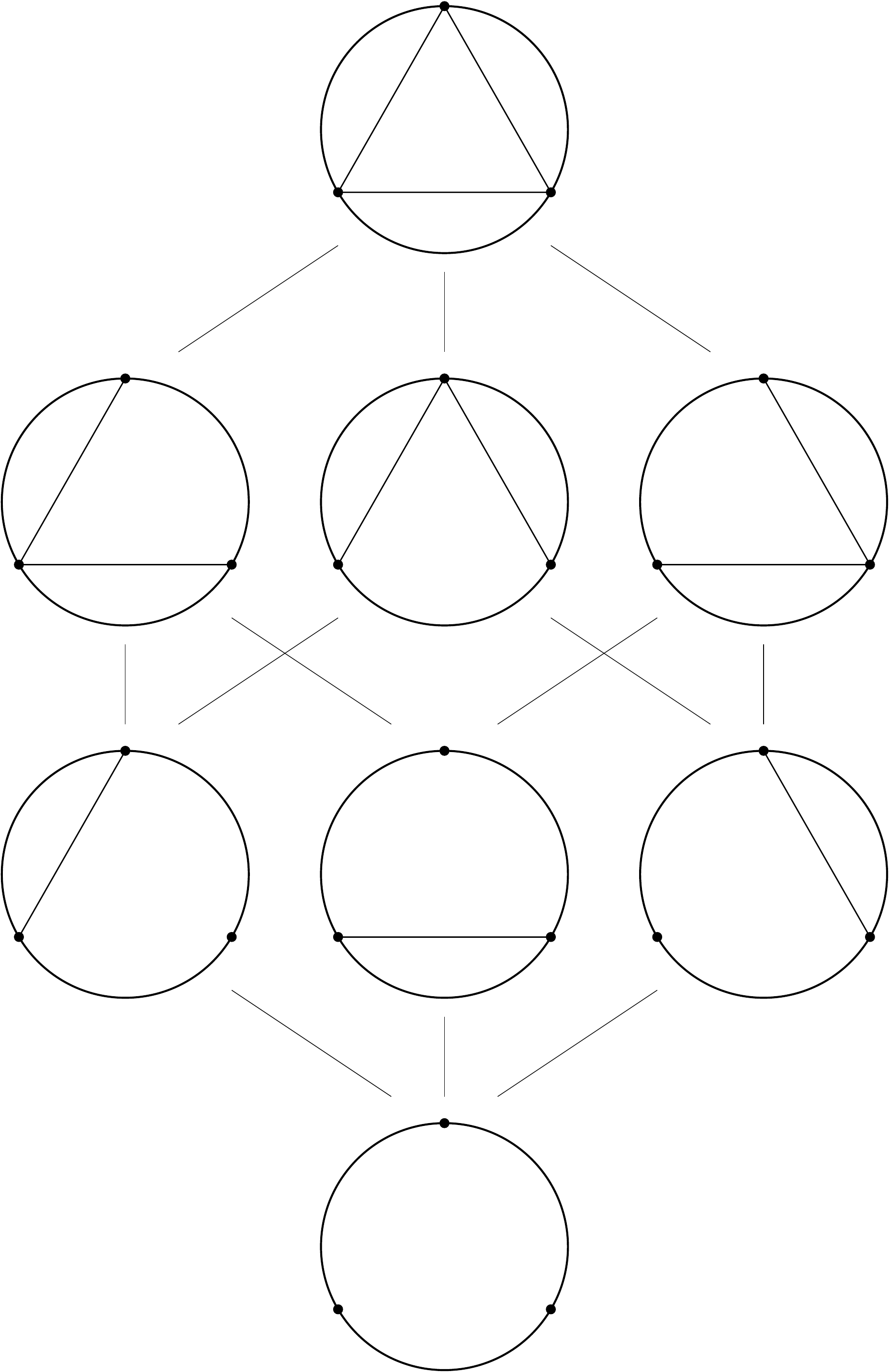}
\caption{$EP_{3}$}\label{EP3}
\end{center}
\end{figure}

Let us now give an alternate description of the poset $EP_{n}$.  Associated to each circular planar graph $G$, we have an open cell $\Omega(G)$ of response matrices for conductances on $G$, where $\Omega(G)$ is taken to be a subset of the space $\Omega_{n}$ of symmetric $n\times n$ matrices. It is clear that, if $G\sim G'$, we have, by definition, $\Omega(G)=\Omega(G')$.

\begin{proposition}\label{closure}
Let $G$ be a circular planar graph. Then, 
\begin{equation}
\overline{\Omega(G)}=\bigsqcup_{H\le G}\Omega(H),
\end{equation}
where $\overline{\Omega(G)}$ denotes the closure of $\Omega(G)$ in $\Omega_{n}$, and the union is taken over equivalence classes of circular planar graphs $H\le G$ in $EP_{n}$.
\end{proposition}

Because the $\Omega(G)$ are pairwise disjoint when we restrict ourselves to equivalence classes of circular planar graphs (a consequence of Theorems \ref{circularminorspositive} and \ref{equivalentconnections}), we get:

\begin{theorem} \label{closureorder}
$[H]\le[G]$ in $EP_{n}$ if and only if $\Omega(H)\subset\overline{\Omega(G)}$.
\end{theorem}

\begin{proof}[Proof of Proposition \ref{closure}]
Without loss of generality, we may take $G$ to be critical. Let $N$ be the number of edges of $G$. By \ref{diffeo}, the map $r_{G}:\mathbb{R}_{>0}^{N}\rightarrow\Omega(G)\subset\Omega_{n}$, sending a collection of conductances of the edges of $G$ the resulting response matrix, is a diffeomorphism. We will describe a procedure for producing a response matrix for any electrical network whose underlying graph $H$ is obtainable from $G$ by a sequence of contractions and deletions (that is, $H\le G$).

Given $\gamma\in\mathbb{R}^{N}_{>0}$, write $\gamma=(\gamma_{1},\ldots,\gamma_{N})$. Note that for each $i\in[1,N]$ and fixed conductances $\gamma_{1},\ldots,\widehat{\gamma_{i}},\ldots,\gamma_{n}$, the limit $\lim_{\gamma_{i}\rightarrow0}r_{G}(\gamma)$ must exist; indeed, sending the conductance $\gamma_{i}$ to zero is equivalent to deleting its associated edge. This fact is most easily seen by physical reasoning: an edge of zero conductance has no current flowing through it, and thus the network may as well not have this edge. Thus, $\lim_{\gamma_{i}\rightarrow0}r_{G}(\gamma)$ is just $r_{G'}(\gamma_{1},\ldots,\widehat{\gamma_{i}},\ldots,\gamma_{n})$, where $G'$ is the result of deleting $e$ from $G$. Similarly, we find that $\lim_{\gamma_{i}\rightarrow\infty}r_{G}(\gamma)$ is $r_{G''}(\gamma_{1},\ldots,\hat{\gamma_{i}},\ldots,\gamma_{n})$, where $G''$ is the result of contracting $e$.

It follows easily, then, that for all $H$ which can be obtained from $G$ by a contraction or deletion, we have $\Omega(H)\subset\overline{\Omega(G)}$, because, by the previous paragraph, $\Omega(H)=\text{Im}(r_{H})\subset\overline{\Omega(G)}$. By induction, we have the same for all $H\le G$.

It is left to check that any $M\in\overline{\Omega(G)}$ is in some cell $\Omega(H)$ with $H\le G$. We have that $M$ is a limit of response matrices $M_{1},M_{2},\ldots\in\Omega(G)$. The determinants of the circular minors of $M$ are limits of determinants of the same minors of the $M_{i}$, and thus non-negative. It follows that $M$ is the response matrix for some network $H$, that is, $M\in\Omega(H)$. We claim that $H\le G$, which will finish the proof.

Consider the sequence $\{C_{k}\}$ defined by $C_{k}=r_{G}^{-1}(M_{k})$, which is a sequence of conductances on $G$. For each edge $e\in G$, we get a sequence $\{C(e)_{k}\}$ of conductances of $e$ in $\{C_{k}\}$. It is then a consequence of the continuity of $r_{G},r_{G}^{-1}$, and the existence of the limits $\lim_{\gamma_{i}\rightarrow0}r_{G}(\gamma),\lim_{\gamma_{i}\rightarrow\infty}r_{G}(\gamma)$, that the sequences $\{C(e)_{k}\}$ each converge to a finite nonnegative limit or otherwise go to $+\infty$.

Furthermore, we claim that for a boundary edge $e$ (that is, one that connects two boundary vertices), $\{C(e)_{k}\}$ cannot tend to $+\infty$. Suppose, instead, that such is the case, that for some boundary edge $e=V_{i}V_{j}$, we have $C(e)_{k}\rightarrow\infty$. Then, note that imposing a positive voltage at $V_{i}$ and and zero voltage at all other boundary vertices sends the current measurement at $V_{i}$ to $-\infty$ as $C(e)_{k}\rightarrow\infty$. In particular, our sequence $M_{1},M_{2},\ldots$ cannot converge, so we have a contradiction.

To finish, it is clear, for example, using similar ideas to the proof of the first direction, that contracting the edges $e$ for which $C(e)_{k}\rightarrow\infty$ (which can be done because such $e$ cannot be boundary edges) and deleting those for which $C(e)_{k}\rightarrow0$ yields $H$. The proof is complete.
\end{proof}

\subsection{Gradedness}

In this section, we prove our first main theorem, that $EP_{n}$ is graded.

\begin{proposition}\label{coverrelation}
$[G]$ covers $[H]$ in $EP_{n}$ if and only if, for a critical representative $G\in[G]$, an edge of $G$ may be contracted or deleted to obtain a critical graph in $[H]$.
\end{proposition}
\begin{proof}

First, suppose that $G$ and $H$ are critical graphs such that deleting or contracting an edge of $G$ yields $H$. Then, if $[G]>[X]>[H]$ for some circular planar graph $X$, some sequence of at least two deletions or contractions of $G$ yields $H'\sim H$. It is clear that $H'$ has fewer edges than $H$, contradicting Theorem \ref{criticalcharacterization}. It follows that $[G]$ covers $[H]$.

We now proceed to prove the opposite direction. Fix a critical graph $G$, and let $e$ be an edge of $G$ that can be deleted or contracted in such a way that the resulting graph $H$ is not critical. By way of Lemmas \ref{constructK} and \ref{constructK'}, we will first construct $T\sim G$ with certain properties, then, from $T$, construct a graph $G'$ such that $[G]>[G']>[H]$. The desired result will then follow: indeed, suppose that $[G]$ covers $[H]$ and $G\in [G]$ is critical. Then, there exists an edge $e\in G$ which may be contracted or deleted to yield $H\in[H]$, and it will also be true that $H$ is critical.

First, we translate to the language of medial graphs. When we break a crossing in the medial graph $\mc{M}(G)$, we may create lenses that must be resolved to produce a lensless medial graph. Suppose that our deletion or contraction of $e\in G$ corresponds to breaking the crossing between the geodesics $ab$ and $cd$ in $\mc{M}(G)$, where the points $a,c,b,d$ appear in clockwise order on the boundary circle. Let $ab\cap cd=p$, and suppose that when the crossing at $p$ is broken, the resulting geodesics are $ad$ and $cb$.

For what follows, let $\mc{F}=\{f_{1},\ldots,f_{k}\}$ denote the set of geodesics $f_{i}$ in $\mc{M}(G)$ such that $f_{i}$ intersects $ab$ between $a$ and $p$, and also intersects $cd$ between $d$ and $p$. We now construct $T$ in two steps.

\begin{lemma}\label{constructK}
There exists a lensless medial graph $K$ such that: 
\begin{itemize}
\item $K$ is equivalent to $\mc{M}(G)$,
\item geodesics $ab$ and $cd$ still intersect at $p$, and breaking the crossing at $p$ to give geodesics $ad$, $bc$ yields a medial graph equivalent to $\mc{M}(H)$, and
\item for $f_i,f_j\in\mc{F}$ which cross each other, the crossing $f_{i}\cap f_{j}$ lies outside the sector $apd$.
\end{itemize}
\end{lemma}
\begin{proof}
The proof is similar to that of \cite[Lemma 6.2]{curtis}. Start with the medial graph $\mc{M}(G)$, and for each $f_i\in\mc{F}$, let $v_i=f_{i}\cap ab$. Also, for each $f_i\in\mc{F}$ which intersects another $f_j\in\mc{F}$ in the sector $apd$, let $D_i$ be the closest point of intersection of some $f_{j}$ along $f_{i}$ to $v_i$ in $apd$. Let $D$ be the set of $D_i$.

If $D$ is empty, there is nothing to check, so we assume that $D$ is nonempty. Then, consider the subgraph of $\mc{M}(G')$ obtained by restricting to the geodesics in $\mc{F}$, along with $ab$ and $cd$. In this subgraph, choose a point $D_i\in D$ such that the number $r$ of regions within the configuration formed by $f_i$, $f_j$, and $ap$ is a minimum, where $f_j$ denotes the other geodesic passing through $D_i$. 

We claim that $r=1$: assume otherwise. Then, there exists a geodesic $f_k$ intersecting $f_j$ between $v_j$ and $D_{i}$ and intersecting $ap$ between $v_i$ and $v_j$, as, by definition, $D_i$ is the first point of intersection on $f_i$ after $v_i$. However, the area enclosed by $f_k$, $f_j$, and $ap$ a number of regions strictly fewer than $r$. Hence, we could instead have chosen the point $D_j\in D$, with $D_{j}\neq D_{i}$, contradicting the minimality.

It follows that $ap$, $f_i$, and $f_j$ form a triangle, and thus the crossing at $D_{i}$ may be moved out of sector $apd$ by a motion. Iterating this process, a finite number of motions may be applied in such so that no $f_i,f_j\in\mc{F}$ intersect in the sector $apd$. After applying these motions, we obtain a medial graph $K$ equivalent to $\mc{M}(G')$ satisfying the first and third properties.

It is easy to see that $K$ also satisfies the second property, as none of the motions involved use the crossing at $p$. Thus, if we translate the sequence of motions into Y-$\Delta$ transformations on circular planar graphs, starting with $G$, no Y-$\Delta$ transformation is applied involving the edge $e$ corresponding to $p$. Thus, deleting or contracting $e$ commutes with the Y-$\Delta$ transformations we have performed.
\end{proof}

It now suffices to consider the graph $K$. Let $f_1\in\mc{F}$ be the geodesic intersecting $ab$ at the point $v_{1}$ closest to $p$, and let $w_1=f_1\cap cd$.

\begin{lemma}\label{constructK'}
There exists a lensless medial graph $K'\sim K$, such that:
\begin{itemize}
\item geodesics $ab$ and $cd$ intersect at $p$, as before, and breaking the crossing at $p$ to give geodesics $ad$, $bc$ yields a medial graph equivalent to $\mc{M}(H)$, and
\item No other geodesic of $K'$ enters the triangle with vertices $v_1,p,w_1$.
\end{itemize}
\end{lemma}
\begin{proof}
We first consider the set $\mc{X}$ of geodesics that only intersect $cd$ and $f_1$. With an argument similar to that of Lemma \ref{constructK}, we may first apply motions so that any intersection of two elements $\mc{X}$ occurs outside the triangle with vertices $v_{1},p,w_{1}$. Then, we may apply motions at $w_{1}$ to move each of the geodesics in $\mc{X}$ outside of this triangle, so that they intersect $f_1$ in the sector $bpd$. After applying similar motions to the set of geodesics $\mc{Y}$ intersecting $ab$ and $f_1$, we have $K'$. The fact that $K'$ satisfies the first desired property follows from the same argument as that of Lemma \ref{constructK'}.
\end{proof}

We are now ready to finish the proof of Proposition \ref{coverrelation}. Let $T=\mc{E}(K')$ (see Theorem \ref{critmedbij}). Then, in $T$, because of the properties of $K'$, contracting $e$ to form the graph $H'\sim H$ forms a pair of parallel edges. Replacing the parallel edges with a single edge gives a circular planar graph $H''$, which is still equivalent to $H$. Suppose that $e$ has endpoints $B,C$ and the edges in parallel are formed with $A$. Then, we have the triangle $ABC$ in $T$.

Write $S=\pi(T)$ (see Definition \ref{connection}) and $S'=\pi(H')$. Because $T$ is critical, $S'\neq S$, so fix $(P;Q)\in S-S'$. Then, it is straightforward to check that any connection $\mc{C}$ between $P$ and $Q$ must have used both $B$ and $C$, but cannot have used the edge $BC$. Furthermore, $\mc{C}$ can use at most one of the edges $AB,AC$. Indeed, if both $AB,AC$ are used, they appear in the same path $\gamma$, but replacing the two edges $AB,AC$ with $BC$ in $\gamma$ gives a connection between $P$ and $Q$, but we know that no such connection can use $BC$, a contradiction. Without loss of generality, suppose that $\mc{C}$ does not use $AB$. Then, deleting $AB$ from $T$ yields a graph $G'$ with $(P;Q)\in G'$, hence $G'$ is not equivalent to $H$. However, it is clear that deleting $BC$ from $G'$ yields $H''\sim H$. It follows, then, that in the case in which $e$ is contracted, we have $G'$ such that $[G]>[G']>[H]$, and hence $[G]$ does not cover $[H]$.

For the case in which we delete $e=ZC$ in $T$, the argument is similar. Deleting $e$ in $T$ yields a graph $H'\sim H$ with two edges $AZ,ZB$ in series, which implies that $T$ has a Y with vertices $A,B,C,Z$, where $Z$ is the middle vertex. It is easy to see that $Z$ is not a boundary vertex. Then, replacing $AZ,ZB$ in $H'$ with the edge $AB$ yields a graph $H''\sim H$. There exists a circular pair $(P;Q)\in\pi(T)-\pi(H)$, so we have a connection $\mc{C}$ between $P$ and $Q$ using the edge $ZC$. Then, $\mc{C}$ also must use exactly one of $AZ$ and $BZ$: wthout loss of generality, assume it is $AZ$. Contracting $BZ$ in $T$ to yield the graph $G'$ leaves $\mc{C}$ intact, and deleting $ZC$ from $G'$ gives $H''\sim H$. As before, we thus have $[G]>[G']>[H]$, so we are done.
\end{proof}

\begin{theorem}\label{graded}
$EP_{n}$ is graded by number of edges of critical representatives.
\end{theorem}

\begin{proof}
First, by Theorem \ref{criticalcharacterization}, note that for any $[G]\in EP_{n}$, all critical representatives of $[G]$ have the same number of edges. Now, we need to show that if $[G]$ covers $[H]$, the number of edges in a critical representative of $[G]$ is one more than the same number for $[H]$. Let $G\in[G]$ be critical. By Proposition \ref{coverrelation}, an edge of $G$ may be contracted or deleted to yield a critical representative $H\in[H]$, and it is clear that $H$ has one fewer edge than $G$.
\end{proof}

\begin{defn}
For all non-negative integers $r$, denote the set of elements of $EP_{n}$ of rank $r$ by $EP_{n,r}$.
\end{defn}

Let us pause to point out connections between $EP_{n}$ and two other posets, interpreting $EP_{n}$ as the graded poset of lensless medial graphs with the covering relation arising from the legal breakings of crossings that preserve lenslessness. 

First, $EP_{n}$ bears a strong resemblance to the symmetric group $S_{n}$ under the (strong) Bruhat order, as follows. Associated to each permutation $\sigma\in S_{n}$, there is a lensless wiring diagram, with $n$ wires connecting two parallel lines $\ell_{1},\ell_{2}$, both with marked points $1,2,\ldots,n$. For each $i\in[n]$, there is a wire joining the point $i\in\ell_{1}$ to $\sigma(i)\in\ell_{2}$. Then, the covering relation in $S_{n}$ is exactly that of $EP_{n}$, except for the fact that each crossing can be broken in exactly one legal way.

Also, consider the poset $W_{n}$ of equivalence classes of lensless wiring diagrams, not necessarily full. Here, the equivalence relation here is generated by motions and resolution of lenses. The order relation arises from breaking of crossings, in a similar way to $EP_{n}$, but we are no longer concerned about the creation of dividing lines. $W_{n}$ can be proven to be graded in a similar way to the proof of Theorem \ref{graded}, and it is furthermore not difficult to check that $EP_{n}$ is in fact an interval in $W_{n}$.

\subsection{Toward Eulerianness}

In this section, we discuss the following conjecture, which we make more detailed in \ref{bigconj}.

\begin{conjecture}\label{eulconj}
$EP_{n}$ is Eulerian.
\end{conjecture}

We first prove that all closed intervals of length 2 in $EP_{n}$ have four elements, which reduces the Eulerianness of $EP_{n}$ to lexicographic shellability.

\begin{lemma}\label{diamond}
Suppose $x\in EP_{n,r-1},z\in EP_{n,r+1}$ with $x<z$. Then, there exist exactly two $y\in EP_{n,r}$ with $x<y<z$.
\end{lemma}
\begin{proof}
Take $x$ and $z$ to be (equivalence classes of) lensless medial graphs. By Theorem \ref{graded}, $x$ may be obtained from $z$ by a sequence of two legal resolutions of crossings. Suppose that $x$ contains the intersecting wires (labeled by their endpoints) $ab$ and $cd$, whose intersection is broken (legally, that is, without creating dividing lines) by instead taking wires $ac,bd$. There are two cases for the next covering relation, from which $x$ results: either one of $ac,bd$ is involved, or a crossing between two new wires is broken.

In the first case, suppose that a crossing between $bd$ and $ef$ is broken to give wires $bf,de$. Up to equivalence under motions, we have one of the two configurations in Figure \ref{breakseq}, constituting subcases A and B. We need to show that, in both cases, there is exactly one other sequence of two legal breakings of crossings, starting from $z$, that gives $x$.

\begin{figure}
\begin{center}
\includegraphics[scale=0.4]{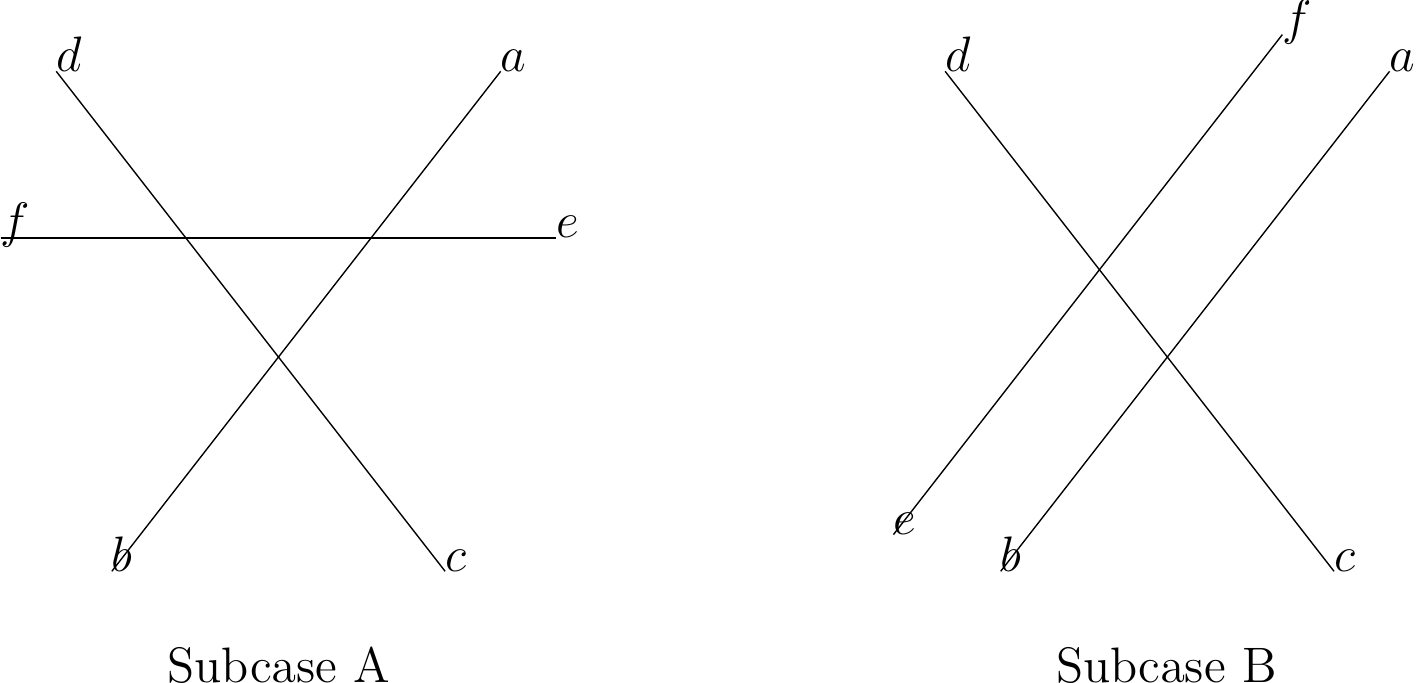}
\caption{Possible starting configurations for two breakings using six medial boundary vertices.}\label{breakseq6}
\end{center}
\end{figure}

In subcase A, there are, at first glance, two possible other ways to get from the set of wires $\{ab,cd,ef\}$ to the set $\{ac,de,bf\}$: the first is through $\{ab,cf,de\}$ and the second is through $\{ae,cd,bf\}$. However, note that the latter case produces a lens, regardless of how the wires are initially positioned to cross each other. Furthermore, assuming the legality of the sequence of breakings $\{ab,cd,ef\}\rightarrow\{ac,bd,ef\}\rightarrow\{ac,bf,de\}$, it is straightforward to check that $\{ab,cd,ef\}\rightarrow\{ab,cf,de\}\rightarrow\{ae,cd,bf\}$ is also a legal sequence of breakings. In subcase B, it is clear that the only other way to get from $z$ to $x$ is through $\{ab,cf,de\}$, and indeed, it is again not difficult to check that we get legal resolutions here.

Now, suppose instead that we have the legal sequence of resolutions 
\begin{equation}\label{breakseq}
\{ab,cd,ef,gh\}\rightarrow\{ac,bd,ef,gh\}\rightarrow\{ac,bd,eg,fh\}
\end{equation}
The only other possible way to get from $z$ to $x$ is through $\{ab,cd,eg,fh\}$. Here, there are a number of cases to check in order to verify legality of the sequence of two breakings involved. The essentially different starting configurations are enumerated in Figure \ref{crossingscommute}. In each, one may check that
\begin{equation}
\{ab,cd,ef,gh\}\rightarrow\{ab,cd,eg,fh\}\rightarrow\{ac,bd,eg,fh\}
\end{equation}
is a sequence of legal breakings of which does not create lenses, which will be a consequence of the fact that the same is true of (\ref{breakseq}). The details are omitted.
\end{proof}

\begin{figure}
\begin{center}
\includegraphics[scale=0.7]{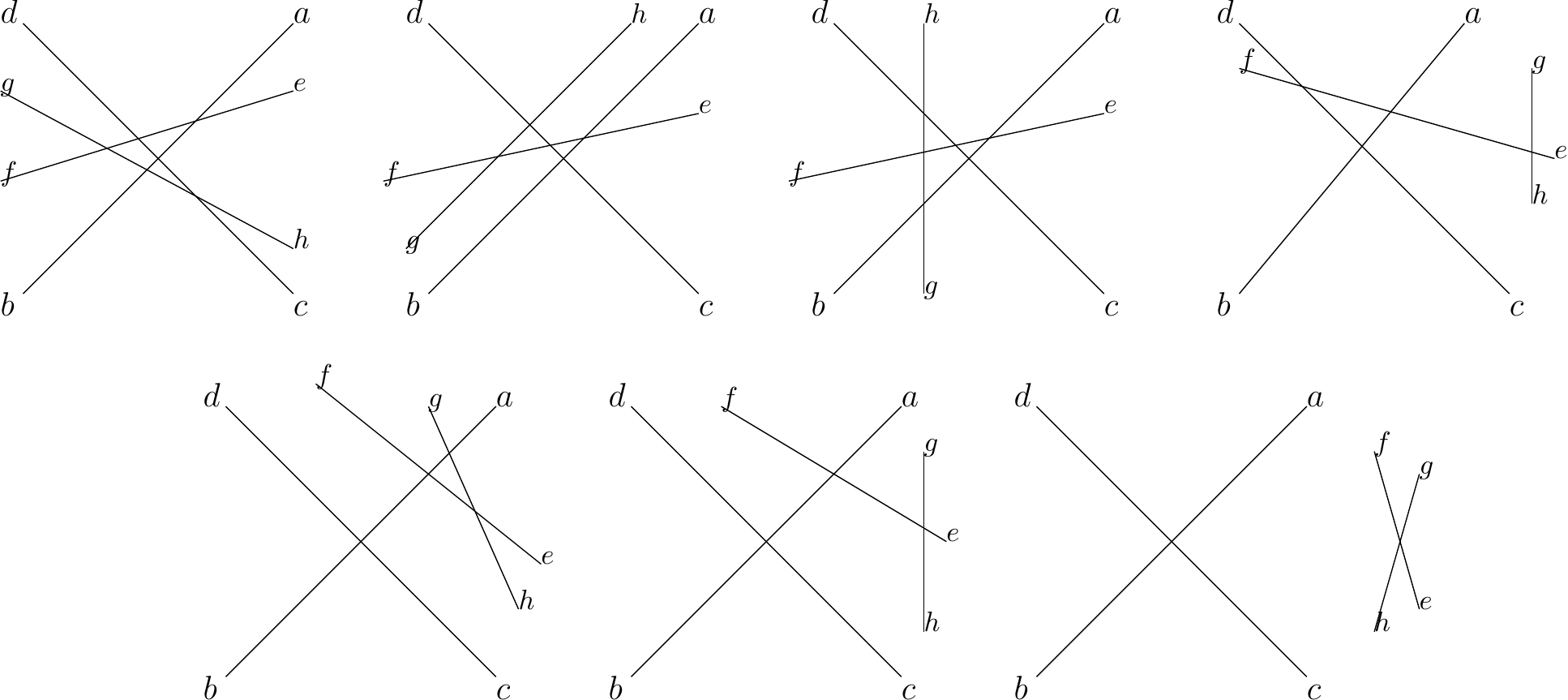}
\caption{Possible starting configurations for two breakings using eight medial boundary vertices.}\label{crossingscommute}
\end{center}
\end{figure}

\begin{conjecture}\label{bigconj}
$EP_{n}$ is lexicographically shellable, and hence Cohen-Macaulay, spherical, and Eulerian.
\end{conjecture}

We refer the reader to \cite{bjornerwachs} for definitions. Indeed, if we have an L-labelling for $EP_{n}$, it would follow that the order complex $\Delta(EP_{n})$ is shellable and thus Cohen-Macaulay (see \cite[Theorem 3.4, Theorem 5.4(C)]{bjornerwachs}). By Lemma \ref{diamond}, \cite[Proposition 4.7.22]{bjorner} would apply, and we would conclude that $EP_{n}$ is spherical and hence Eulerian. 

Through \cite{sage}, $EP_{n}$ has been verified to be Eulerian for $n\le 7$, and the homology of $EP_{n}-\{\widehat{0},\widehat{1}\}$ agrees with that of a sphere of the correct dimension, $\binom{n}{2}-2$, for $n\le 4$. On the other hand, no L-labeling of $EP_{n}$ is known for $n\ge 4$.

It is also worth mentioning the following conjecture concerning the poset $W_{n}$ (defined at the end of the previous section), which implies Conjecture \ref{eulconj}.

\begin{conjecture}
The poset $W_{n}\cup\{\widehat{0}\}$, obtained by adjoining a minimal element to $W_{n}$, is Eulerian.
\end{conjecture}

\section{Enumerative Properties}\label{enumerative}

We now investigate the enumerative properties of $EP_{n}$, defined in \S\ref{poset}. In the sections that follow, all wiring diagrams are assumed to be lensless, and are considered up to motion-equivalence.

\subsection{Total size $X_{n}=|EP_{n}|$}

In this section, we adapt methods of \cite{callan} to prove the first two enumerative results concerning $|EP_{n}|$, the number of equivalence classes of critical graphs (equivalently, full wiring diagrams) of order $n$. There is a strong analogy between stabilized-interval free (SIF) permuations, as described in \cite{callan}, and our medial graphs, as follows. A permutation $\sigma$ may be represented as a 2-regular graph $\Sigma$ embedded in a disk with $n$ boundary vertices. Then, $\sigma$ is SIF if and only if there are no dividing lines, where here a dividing line is a line $\ell$ between two boundary vertices such that no edge of $\Sigma$ connects vertices on opposite sides of $\ell$.

To begin, we define two operations on wiring diagrams in order to build large wiring diagrams out of small, and vice versa. In both definitions, fix a lensless (but not necessarily full) wiring diagram $M$ of order $n$, with boundary vertices labeled $V_{1},V_{2},\ldots,V_{n}$.

\begin{defn}
Let $w=XY$ be a wire of $M$. Construct the \textbf{crossed expansion of $M$ at $w$}, denoted $M^{w}_{+,c}$, as follows: add a boundary vertex $V_{n+1}$ to $M$, with associated medial boundary vertices $A,B$, so that the medial boundary vertices $A,B,X,Y$ appear in order around the circle. Then, delete $w$ from $M$ and replace it with the crossing wires $AX,BY$ to form $M^{w}_{+,c}$. Similarly, define the \textbf{uncrossed expansion of $M$ at $w$}, denoted $M^{w}_{+,u}$, to be the lensless wiring digram obtained by deleting $w$ and replacing it with the non-crossing wires $AY,BX$.
\end{defn}

\begin{defn}\label{refine}
Let $V_{i}$ be a boundary vertex with associated medial boundary vertices $A,B$, such that we have the wires $AX,BY\in M$, and $X\neq B,Y\neq A$. Define the \textbf{refinement of $M$ at $V_{i}$}, denoted $M^{i}_{-}$, to be the lensless wiring diagram of order $n-1$ obtained by deleting the wires $AX,BY$ as well as the vertices $A,B,V_{i}$, and adding the wire $XY$.
\end{defn}

Each construction is well-defined up to equivalence under motions by Theorem \ref{motionsenough}. It is clear that expanding $M$, then refining the result at the appropriate vertex, recovers $M$. Similarly, refining $M$, then expanding the result after appropriately relabeling the vertices, recovers $M$ if the correct choice of crossed or uncrossed is made.

\begin{lemma}\label{explemma}
Let $M$ be a full wiring diagram, with boundary vertices $V_{1},V_{2},\ldots,V_{n}$. Then:
\begin{enumerate}
\item[(a)] $M_{+,c}^{w}$ is full for all wires $w\in M$. 
\item[(b)] Either $M_{+,u}^{w}$ is full, or otherwise $M_{+,u}^{w}$ has exactly one dividing line, which must have $V_{n+1}$ as one of its endpoints.
\end{enumerate}
\end{lemma}

\begin{proof}
First, suppose for sake of contradiction that $M_{+,c}^{w}$ has a dividing line $\ell$. If $\ell$ is of the form $V_{i}V_{n+1}$, then $\ell$ must exit the sector formed by the two crossed wires coming out of the medial boundary vertices associated to $V_{n+1}$. If this is the case, we get an intersection between $M_{+,c}$ and a wire, a contradiction. If instead, $\ell=V_{i}V_{j}$ with $i,j\neq n+1$, then $\ell$ is a dividing line in $M$, also a contradiction. We thus have (a). Similarly, we find that any dividing line of $M_{+,u}^{w}$ must have $V_{n+1}$ as an endpoint. However, if $V_{i}V_{n+1},V_{i'}V_{n+1}$ are dividing lines, then $V_{i}V_{i'}$ is as well, a contradiction, so we have (b).
\end{proof}

\begin{lemma}\label{maxdivline}
Let $M$ be a full wiring diagram, with boundary vertices $V_{1},V_{2},\ldots,V_{n}$. Furthermore, suppose $M^{n}_{-}$ exists and is not full. Then, $M^{n}_{-}$ has a unique dividing line $V_{i}V_{j}$ with $1\le i<j\le n-1$ and $j-i$ maximal.
\end{lemma}

\begin{proof}
By assumption, $M^{n}_{-}$ has a dividing line, so suppose for sake of contradiction that $\ell_{1}=V_{i_{1}}V_{j_{1}},\ell_{2}=V_{i_{2}}V_{j_{2}}$ are both dividing lines of $M'$ with $d=j_{1}-i_{1}=j_{2}-i_{2}$ maximal. Without loss of generality, assume $i_{1}<i_{2}$ (and $i_{1}<j_{1},i_{2}<j_{2}$). If $j_{1}\ge i_{2}$, then $V_{i_{1}}V_{j_{2}}$ is also a dividing line with $j_{2}-i_{1}>d$, a contradiction. On the other hand, if $j_{1}<i_{2}$, at least one of $\ell_{1},\ell_{2}$ is a dividing line for $M$, again a contradiction.
\end{proof}

If $M,i,j$ are as above, we now define two wiring diagrams $M_{1}$ and $M_{2}$; see Figure \ref{M1M2} for an example. First, let $M_{1}$ be the result of restricting $M$ to the wires associated to the vertices $V_{k}$, for $k\in[i,j]\cup\{n\}$. Note that $M_{1}$ is a wiring diagram of order $j-i+1$ with boundary vertices $V_{i},V_{i+1},\ldots,V_{j}$ (and not $V_{n}$). Then, let $M_{2}$ be the wiring diagram of order $n-(j-i+1)$ obtained by restricting $M$ to the wires associated to the vertices $V_{k}$, for $k\notin[i,j]\cup\{n\}$.

\begin{figure}
\begin{center}
\includegraphics[scale=0.4]{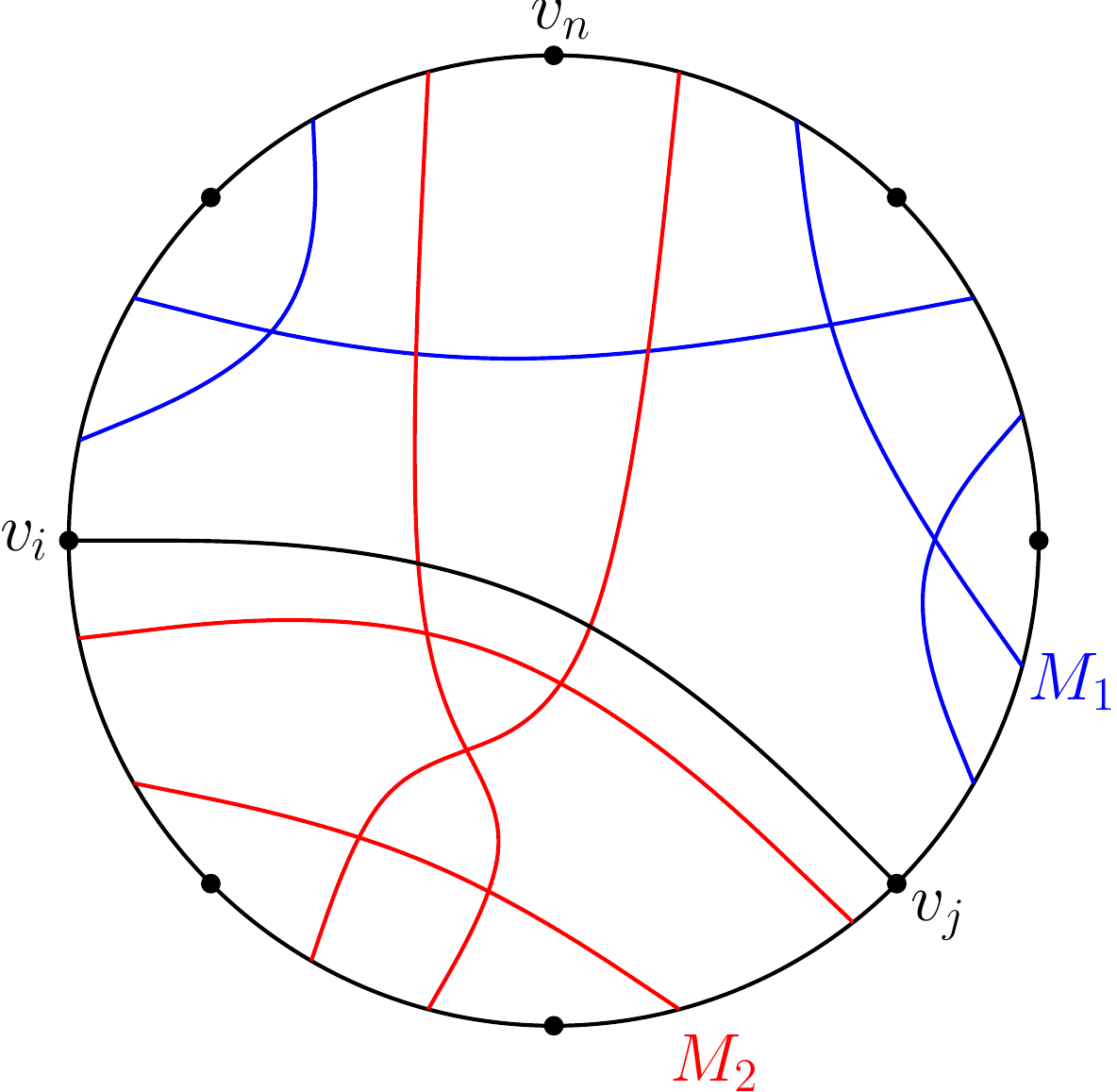}
\caption{$M_{1}$ and $M_{2}$, from $M$.}\label{M1M2}
\end{center}
\end{figure}

\begin{lemma}\label{reflemma}
$M_{1}$ and $M_{2}$, as above, are full.
\end{lemma}

\begin{proof}
It is not difficult to check that any dividing line of $M_{1}$ must also be a dividing line of $M$, a contradiction. A dividing line $V_{i'}V_{j'}$ of $M_{2}$ must also be a dividing line of $M^{n}_{-}$, but then $j'-i'>j-i$, contradicting the maximality from Lemma \ref{maxdivline}.
\end{proof}

We are now ready to prove the main theorem of this section.

\begin{theorem}\label{recurrence}
Put $X_{n}=|EP_{n}|$, which here we take to be the number of full wiring diagrams of order $n$. Then, $X_{1}=1$, and for $n\ge2$,
\begin{equation*}
X_{n}=2(n-1)X_{n-1}+\sum_{k=2}^{n-2}(k-1)X_{k}X_{n-k}.
\end{equation*}
\end{theorem}

\begin{proof}
$X_{1}=1$ is obvious. For $n>1$, we would like to count the number of full wiring diagrams $M$ of order $n$, whose boundary vertices are labeled $V_{1},V_{2},\ldots,V_{n}$, in clockwise order, with medial boundary vertices $A_{i}$ and $B_{i}$ at each vertex, so that the order of points on the circle is $A_{i},V_{i},B_{i}$ in clockwise order. If $A_{n}B_{n}$ is a wire, constructing the rest of $M$ amounts to constructing a full wiring diagram of order $n-1$, so there are $X_{n-1}$ such full wiring diagrams in this case.

Otherwise, consider the refinement $M^{n}_{-}$. All $M$ for which $M^{n}_{-}$ is full can be obtained by expanding at one of the $n-1$ wires of a full wiring diagram $M'$ of order $n-1$. By Lemma \ref{explemma}, the expanded wiring diagram is full unless it has exactly one dividing line $V_{k}V_{n}$, and furthermore it is easy to see that any such graphs is an expansion of a full wiring diagram of order $n-1$.

There are $2(n-1)$ ways to expand $M'$, and each expansion gives a different wiring diagram of order $n$, for $2(n-1)X_{n-1}$ total expanded wiring diagrams. However, by the previous paragraph, the number of these which are not full is $\sum_{k=1}^{n-1}X_{k}X_{n-k}$, as imposing a unique dividing line $V_{k}V_{n}$ forces us to construct two full wiring diagrams on either side, of orders $k,n-k$ respectively. Thus, we have $2(n-1)X_{n-1}-\sum_{k=1}^{n-1}X_{k}X_{n-k}$ full wiring diagrams of order $n$ such that refining at $V_{n}$ gives another full wiring diagram.

It is left to count those $M$ such that contracting at $V_{n}$ leaves a non-full wiring diagram $M'$. By Lemma \ref{reflemma}, such an $M$ gies us a pair of full wiring diagrams of orders $i+j+1,n-(i+j+1)$, where $V_{i}V_{j}$ is as in Lemma \ref{maxdivline}. Conversely, given a pair of boundary vertices $V_{i},V_{j}\neq V_{n}$ of $M$ and full wiring diagrams of orders $j-i+1,n-(j-i+1)$, we may reverse the construction $M\mapsto(M_{1},M_{2})$ to get a wiring diagram of order $n$: furthermore, it is not difficult to check that this wiring diagram is full.

It follows that the number of such $M$ is 
\begin{equation*}
\sum_{1\le i<j\le n-1}X_{j-i+1}X_{n-(j-i+1)}=\sum_{k=1}^{n-2}kX_{k}X_{n-k}.
\end{equation*}
Summing our three cases together, we find
\begin{align*}
X_{n}&=X_{n-1}+2(n-1)X_{n-1}-\sum_{k=1}^{n-1}X_{k}X_{n-k}+\sum_{k=1}^{n-2}kX_{k}X_{n-k}\\
&=2(n-1)X_{n-1}+\sum_{k=2}^{n-2}(k-1)X_{k}X_{n-k},
\end{align*}
using the fact that $X_{1}=1$. The theorem is proven. 
\end{proof}

\begin{rmk}The sequence $\{X_{n}\}$ is found in the Online Encyclopedia of Integer Sequences, see \cite{oeis}.
\end{rmk}

We also have an analogue of the other main result of \cite{callan}.

\begin{theorem}\label{gf}
Let $X(t)=\sum_{n=0}^{\infty}X_{n}t^{n}$ be the generating function for the sequence $\{X_{n}\}$, where we take $X(0)=0$. Then, we have $[t^{n-1}]X(t)^{n}=n\cdot(2n-3)!!$.
\end{theorem}

\begin{proof}

Consider $n$ boundary vertices on a circle, labeled $V_{1},V_{2},\ldots,V_{n}$ in clockwise order. Then, label $2n$ medial boundary vertices $W_{1},W_{2},\ldots,W_{2n}$ in clockwise order so that $W_{2n-1}$ and $W_{2n}$ lie between $V_{n}$ and $V_{1}$ on the circle. Note that $n\cdot(2n-3)!!$ counts the number of wiring diagrams so that the wire with endpoint $W_{2n}$ has second endpoint $W_{z}$, for some $z$ odd. Call such wiring diagrams \textbf{$2n$-odd}. We need a bijection between $2n$-odd wiring diagrams and lists of $n$ full wiring diagrams with sum of orders equal to $n-1$.

From here, the rest of the proof is nearly identical to the analogous result given on \cite[p. 3]{callan}, so we give only a sketch. We will refer the reader often to \cite{callan} for more details.

Let $W$ be a $2n$-odd wiring diagram, with boundary vertices and medial boundary vertices labeled as above. For $i=1,2,\ldots,n$, let $p_{i}$ denote the pair of medial boundary vertices $\{W_{2i-1},W_{2i}\}$. Consider the set of dividing lines of $W$. We first partition the $p_{i}$ in to minimal consecutive blocks $I=\{p_{k},p_{k+1},\ldots,p_{\ell}\}$, where indices are \emph{not} taken modulo $n$, such that no wire has one endpoint in some $p_{i}\in I$ and the other in some $p_{j}\notin I$. Let $\pi$ denote this partition, with blocks $\pi_{1},\pi_{2},\ldots,\pi_{d}$. We order the blocks in such a way that if $p_{i}\in\pi_{a}$, and $p_{j}\in\pi_{b}$, then, if $i<j$, we have $a<b$. Note that, in particular, $p_{n}\in\pi_{d}$.

Now, for each block $\pi_{a}$, write $|\pi_{a}|=x_{a}$. For $a<d$, $\pi_{a}$ may be further partitioned in to a non-crossing partition of total size $s$, according to the dividing lines in the corresponding subgraph of $W$. Each such partition corresponds to a Dyck path $\mc{P}_{a}$ of length $2x_{a}$, by a bijection described in \cite{callan}, and it is not difficult to check that, because $\pi_{a}$ was constructed to be a minimal connected component, $\mc{P}_{a}$ only touches the $x$-axis at its endpoints.

On $\pi_{d}$, we first perform the following operation similar to refinement, as in Definition \ref{refine}. Let the second endpoints of the wires $w_{\alpha},w_{\beta}$ coming from $W_{2n-1},W_{2n}$, respectively, be $W_{\alpha},W_{\beta}$, respectively. Then, delete the wires $w_{\alpha},w_{\beta}$, and replace them with a single wire between $W_{\alpha},W_{\beta}$. If, however, $W_{2n-1},W_{2n}$ are connected by a single wire, simply delete this wire. In either case, the resulting block $\pi_{d}'$ now has order $x_{d}-1$, and it, too, may be further partitioned in to a non-crossing partition, corresponding to a Dyck path $\mc{P}_{d}$ of length $2x_{d}-2$. Unlike $\mc{P}_{a}$, with $a<d$, $\mc{P}_{d}$ may touch the $x$-axis more than twice.

We now cut the Dyck paths $\mc{P}_{a}$ in a similar way to that of \cite{callan}. For $a<d$, we cut $\mc{P}_{a}$ in the following way: remove the last upstep $u$, thus breaking $\mc{P}_{a}$ in to a path $P_{a}$, followed by an upstep $u$, and then followed by a descent $D_{a}$. As for $\mc{P}_{d}$, recall that, due to the bijection between non-crossing partitions and Dyck paths, the upsteps in $\mc{P}_{d}$ correspond to to the elements $p_{i}\in\pi_{d}$. Let $u_{0}$ denote the upstep corresponding to the $p_{i}$ containing $W_{z}$, the second endpoint of the wire with endpoint $W_{2n}$. Then, break $\mc{P}_{d}$ in to the paths $R,S$, where $R$ is the part of $\mc{P}_{d}$ appearing before $u_{0}$, and $S$ consists of alls steps after those of $R$. In the case that $z=2k-1$, note that $x_{d}=1$ and thus $\mc{P}_{d}$ is empty. In this case, $R,S$ are also taken to be empty.

Finally form the concatenated path $D_{1}uD_{2}u\cdots D_{k-1}uSRP_{1}P_{2}\cdots P_{k-1}$, as in \cite{callan}. There are $n-1$ upsteps in this path, which begins and ends on the $x$-axis. In between these $n-1$ upsteps are $n$ (possibly empty) descents, which, using the bijection between non-crossing partitions and Dyck paths, correspond to full sub-wiring diagrams of the $p_{a}$. Therefore, we get the desired list of $n$ full wiring diagrams of total order $n-1$, and the process is reversible by an argument similar to that of \cite{callan}. The details are left to the reader. 

\end{proof}

\subsection{Asymptotic Behavior of $X_{n}=|EP_{n}|$}\label{asymptoticsection}

In this section, we adapt methods from \cite{stein} to prove:

\begin{theorem}\label{asymptotic}
We have
\begin{equation*}
\lim_{n\rightarrow\infty}\frac{X_{n}}{(2n-1)!!}=\frac{1}{\sqrt{e}}.
\end{equation*}
\end{theorem}

In other words, the density of full wiring diagrams in the set of all wiring diagrams is $e^{-1/2}$.

\begin{lemma}\label{ratio}
For $n\ge6$, $(2n-1)X_{n-1}<X_{n}<2nX_{n-1}$.
\end{lemma}

\begin{proof}
We proceed by strong induction on $n$: the inequality is easily verified for $n=6,7,8$ using Theorem \ref{recurrence}. Furthermore, note that $X_{n}<2nX_{n-1}$ for $n=2,3,4,5$ as well. Now, assume $n\ge9$.

By Lemma \ref{recurrence}, it is enough to show 
\begin{equation}\label{Xand2X}
X_{n-1}<\sum_{j=2}^{n-2}(j-1)X_{j}X_{n-j}<2X_{n-1}.
\end{equation}

We first show the left hand side of (\ref{Xand2X}). Now, we have
\begin{align*}
\sum_{j=2}^{n-2}(j-1)X_{j}X_{n-j}&>X_{2}X_{n-2}+(n-4)X_{n-3}X_{3}+(n-3)X_{n-2}X_{2}\\
&=2(n-2)X_{n-2}+8(n-4)X_{n-3}\\
&>\left(\frac{n-2}{n-1}+\frac{2(n-4)}{(n-2)(n-1)}\right)X_{n-1}\\
&>X_{n-1},
\end{align*}
where we have applied the inductive hypothesis.

It remains to prove the right hand side of (\ref{Xand2X}). First, suppose that $n$ is odd, with $n=2k-1,k\ge4$. Let $Q_{i}=X_{i}/X_{i-1}$ for each $i$; we know that $Q_{i}>2i-1$ for all $i\ge5$. Then, we have

\begin{align*}
\sum_{j=2}^{n-2}(j-1)X_{j}X_{n-j}&=(2k-3)\sum_{j=2}^{k-1}X_{j}X_{2k-1-j}\\
&=(2k-3)X_{n-1}\sum_{j=2}^{k-1}\frac{X_{j}}{Q_{2k-2}Q_{2k-3}\cdots Q_{2k-j}}\\
&<(2k-3)X_{n-1}\sum_{j=2}^{k-1}\frac{X_{j}}{(4k-5)(4k-7)\cdots(4k-2j-1)}
\end{align*}
However, we claim that the terms in the sum are strictly decreasing. This amounts to the inequality $(4k-2j-1)X_{j-1}>X_{j}$ for $3\le j\le k-1$, which follows by the inductive hypothesis as $4k-2j-1>2j$. Thus, 
\begin{align*}
&(2k-3)X_{n-1}\sum_{j=2}^{k-1}\frac{X_{j}}{(4k-5)(4k-7)\cdots(4k-2j-1)}\\
<\ &(2k-3)X_{n-1}\left(\frac{X_{2}}{4k-5}+\frac{(k-3)X_{3}}{(4k-5)(4k-7)}\right)\\
=\ &X_{n-1}\left(\frac{4k-6}{4k-5}+\frac{(4k-12)(4k-6)}{(4k-5)(4k-7)}\right)\\
<\ &2X_{n-1},
\end{align*}
where we substitute $X_{2}=2,X_{3}=8$. The case in which $n$ is even may be handled similarly, and the induction is complete.
\end{proof}

\begin{corollary}\label{limit}
There exists a limit 
\begin{equation*}
C=\lim_{n\rightarrow\infty}\frac{X_{n}}{(2n-1)!!},
\end{equation*}
and furthermore, $C>0$.
\end{corollary}

\begin{proof}
The sequence $X_{n}/(2n-1)!!$ is bounded above by 1 and is eventually strictly increasing by Lemma \ref{ratio}, so the limit $C$ exists. Furthermore, $C>0$ because $X_{n}/(2n-1)!!$ is eventually increasing.
\end{proof}

To prove Theorem \ref{asymptotic}, we will estimate the number of non-full wiring graphs of order $n$. Let $D_{n}$ denote the number of wiring diagrams formed in the following way: for $1\le j\le n-2$, choose $j$ pairs of adjacent boundary vertices, and for each pair, connect the two medial boundary vertices between them. Then, with the remaining $2n-2j$ vertices, form a full wiring diagram of order $n-j$, which in particular has no dividing lines whose endpoints are adjacent boundary vertices. It is clear that all such diagrams are non-full.

For completeness, we will also include in our count the wiring diagram where all pairs of adjacent boundary vertices give dividing lines, but because we are interested in the asymptotic behavior of $D_{n}$, this addition will be of no consequence. It is easily seen that 
\begin{equation*}
D_{n}=1+\sum_{j=1}^{n-2}\binom{n}{j}X_{n-j}.
\end{equation*}

Now, let $E_{n}$ be the number of non-full wiring diagrams not constructed above. Consider the following construction: choose an ordered pair of distinct, non-adjacent boundary vertices on our boundary circle. Then, on each side of the directed segment, construct any wiring diagram. This construction yields 
\begin{equation*}
Y_{n}=n\sum_{j=2}^{n-2}(2n-2j-1)!!(2j-1)!!
\end{equation*}
total (not necessarily distinct) wiring diagrams, which clearly overcounts $E_{n}$.

We now state two lemmas:

\begin{lemma}\label{dbound}
$D_{n}/X_{n}\to \sqrt{e}-1$ as $n\to\infty$
\end{lemma}

\begin{lemma}\label{ebound}
$Y_{n}/X_{n}\to 0$ as $n\to\infty$.
\end{lemma}

From here, we will be able to establish the desired asymptotic.

\begin{proof}[Proof of Theorem \ref{asymptotic}]
$X_{n}$, $D_{n}$, and $E_{n}$ together count the total number of wiring diagrams, which is equal to $(2n-1)!!$. Thus, 
\begin{equation*}
\frac{(2n-1)!!}{X_{n}}=\frac{X_{n}+D_{n}+E_{n}}{X_{n}}\to1+(\sqrt{e}-1)+0=e^{1/2},
\end{equation*}
assuming Lemmas \ref{dbound} and \ref{ebound} (we have $Y_{n}/X_{n}\to0$, so $E_{n}/X_{n}\to0$ as well), so the desired conclusion is immediate from taking the reciprocal.
\end{proof}

Thus, it remains to prove Lemmas \ref{dbound} and \ref{ebound}, which we defer to Appendix \ref{bounds}.

Let us summarize now the results of the last two sections:

\begin{theorem}\label{enumsummary}
\begin{enumerate}
\item[(a)] $X_{1}=1$ and 
\begin{equation*}
X_{n}=2(n-1)X_{n-1}+\sum_{j=2}^{n-2}(j-1)X_{j}X_{n-j}.
\end{equation*}
\item[(b)] $[t^{n-1}]X(t)^{n}=n\cdot(2n-3)!!$, where $X(t)$ is the generating function for the sequence $\{X_{i}\}$.
\item[(c)] $X_{n}/(2n-1)!! \to e^{-1/2}$ to $\infty$.
\end{enumerate}
\end{theorem}

To conclude this section, we propose the following generalization of Theorem \ref{enumsummary}:

\begin{conjecture}
Let $\lambda$ be a positive integer. Consider the sequence $\{X_{n,\lambda}\}$ defined by $X_{1,\lambda}=1$, and 
\begin{equation*}
X_{n}=\lambda(n-1)X_{n-1,\lambda}+\sum_{k=2}^{n-2}(j-1)X_{j,\lambda}X_{n-k,\lambda}.
\end{equation*}
Then, let $X_{\lambda}(t)$ be the generating function for the sequence $\{X_{n,\lambda}\}$. Then, 
\begin{equation*}
[t^{n-1}]X_{\lambda}(t)^{n}=n\cdot(1/\lambda)_{n}
\end{equation*}
and
\begin{equation*}
\lim_{n\rightarrow\infty}\frac{X_{\lambda,n}}{(1/\lambda)_{n}}=\frac{1}{\sqrt[n]{e}},
\end{equation*}
where $(a)_{n}=a(a+1)\cdots(a+(n-1))$.
\end{conjecture}

A proof exhibiting and exploiting a combinatorial interpretation for the sequence $\{X_{n,\lambda}\}$ would be most desirable, as we have done for $\lambda=2$. However, no such interpretation is known for $\lambda>2$. The case $\lambda=1$ is handled in \cite{callan} and \cite[\S 3]{salvatore}, though the latter does not use the interpretation of $X_{n,1}$ as SIF permutations of $[n]$ to obtain the asymptotic.

Interestingly, if we define $X_{n,-1}$ analogously, we get $X_{n,-1}=(-1)^{n+1}C_{n}$, where $C_{n}$ denotes the $n$-th Catalan number, see \cite{oeis}. 

\subsection{Rank sizes $|EP_{n,r}|$}

\begin{proposition}\label{highrank}
For non-negative $c\le n-2$, we have $|EP_{n,\binom{n}{2}-c}|=\binom{n-1+c}{c}$. Furthermore, $|EP_{n,\binom{n}{2}-(n-1)}|=\binom{2n-2}{n-1}-n$.
\end{proposition}
\begin{proof}For convenience, put $N=\binom{n}{2}$. We claim that for $c\le n-2$, any wiring diagram of order $n$ with $N-c$ crossings is necessarily full. Suppose instead that we have a dividing line, dividing our circle in to two wiring diagrams of orders with $j,n-j$. Then, there are at most 
\begin{equation*}
\binom{j}{2}+\binom{n-j}{2}\le\binom{n-1}{2}=N-(n-1)
\end{equation*}
crossings, so if $c\le n-2$ we cannot have a dividing line.

Thus, for $c\le n-2$, it suffices to compute the number of circular wiring diagrams with $N-c$ crossings. By \cite[(1)]{riordan}, this number is the coefficient of the $q^{N-c}$ term of the polynomial 
\begin{equation}\label{crossinggf}
T_{n}(q)=(1-q)^{-n}\sum_{j=0}^{n}(-1)^{j}\left[\binom{2n}{n-j}-\binom{2n}{n-j-1}\right]q^{\binom{j+1}{2}},
\end{equation}
which, as noted in \cite[p. 218]{riordan}, is $\binom{n+c-1}{n-1}$ for $c\le n-1$. This immediately gives the desired result for $c\le n-2$.

For $c=n-1$, we have, by the above, $\binom{2n-2}{n-1}$ wiring diagrams with $N-c$ crossings; we need to count the number of such wiring diagrams that contain a dividing line. However, note that if our dividing line separates the circle in to wiring diagrams of orders $j,n-j$ for $1<j\le n/2$, there are at most $\binom{n-2}{2}+1$ crossings (using a similar argument to that in the first paragraph), which is strictly less than $N-(n-1)$, so we must have $j=1$.

Furthermore, by the first paragraph, if $j=1$, we need exactly $\binom{n-1}{2}$ crossings. Thus, a non-full wiring diagram with $N-(n-1)$ crossings must connect two adjacent medial boundary vertices between two boundary vertices, and connect all of the other medial boundary vertices in such the unique way such that we have the maximal possible number of crossings between the $n-1$ wires. There are clearly $n$ such non-full wiring diagrams, giving $|EP_{n,N-(n-1)}|=\binom{2n-2}{n-1}-n$, as desired.
\end{proof}

Proposition \ref{highrank} gives an exact formula for $|EP_{n,r}|$ for $r$ large, but no general formula is known for general $r$. For fixed $r$ and $n$ sufficiently large, one will only have finitely many cases to enumerate for possible configurations of an electrical network, but the casework becomes cumbersome quickly. However, the M\"{o}bius Inversion Formula gives us an expression for the generating function for the number of full wiring diagrams of order $n$, counted by number of crossings. 

Let $NC_{n}$ be the (graded) poset of non-crossing partitions on $n$, ordered by refinement. By \cite[Proposition 2.3]{blass}, we have $\mu(\widehat{0},\widehat{1})=(-1)^{n-1}C_{n-1}$ in $NC_{n}$, and furthermore, for any $\pi\in NC_{n}$, the interval $(\widehat{0},\pi)$ is isomorphic to a product of the partition lattices $NC_{k}$, where $k$ ranges over the block sizes of $\pi$. Given $\pi\in NC_{n}$, $\pi$ may be represented as a set of dividing lines in a disk $D$ with boundary vertices $V_{1},V_{2},\ldots,V_{n}$ as follows: draw the dividing line $V_{i}V_{j}$ if $i,j$ are in the same block of $\pi$. Furthermore, the set of dividing lines for a wiring diagram yields a non-crossing partition $[n]$ in the same way.

\begin{figure}
\begin{center}
\includegraphics[scale=0.4]{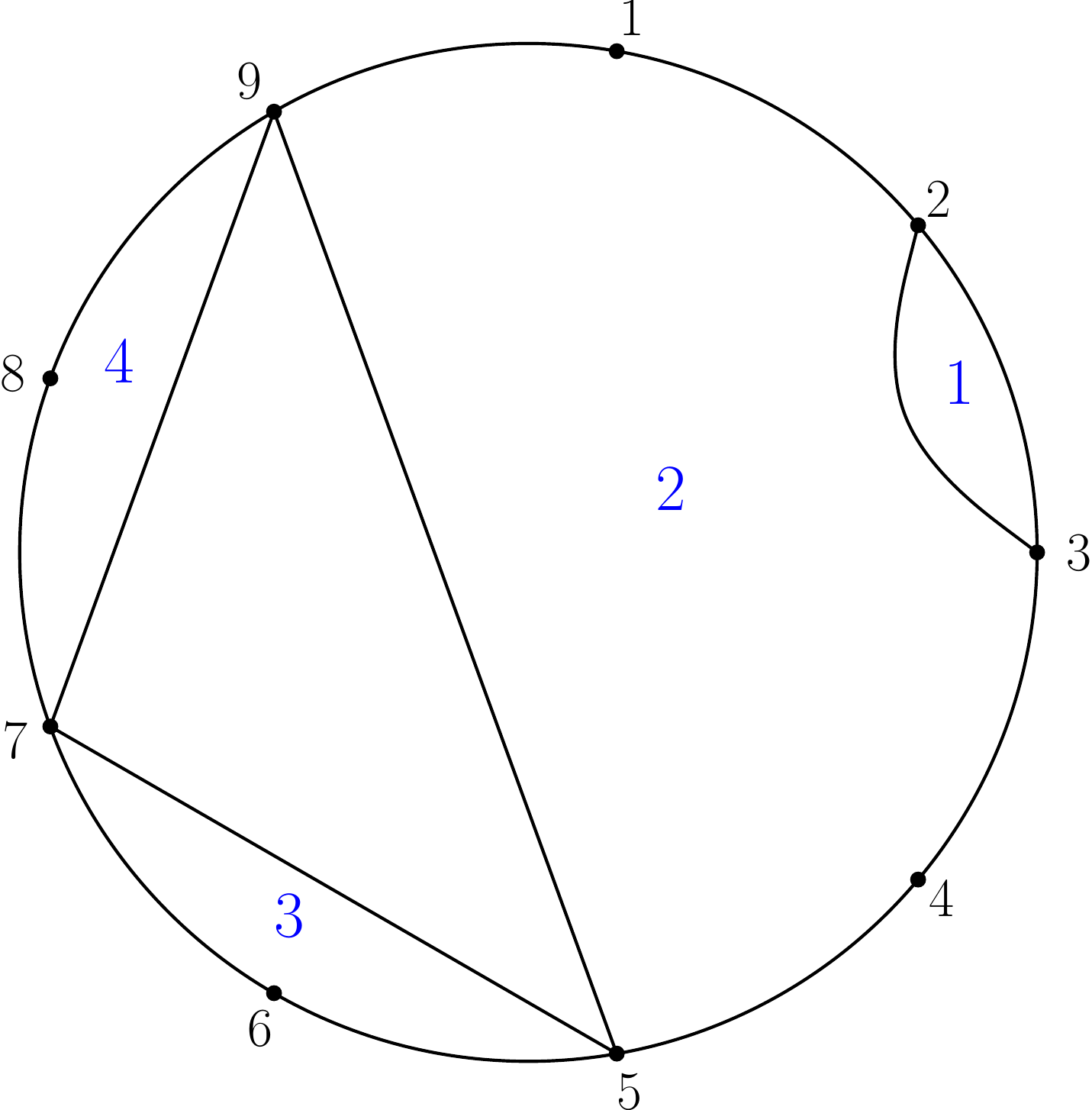}
\caption{$[1][23][4][579][6][8]$ breaks the disk in to four wiring regions.}\label{noncrossing}
\end{center}
\end{figure}

Let $k_{\pi}$ denote the number of blocks in $\pi$. It is clear that drawing these dividing lines of $\pi$ breaks $D$ in to $n+1-k_{\pi}$ regions in which wires can be drawn (see Figure \ref{noncrossing} for an example). Let $a_{\pi,1},\ldots,a_{\pi,n+1-k\pi}$ denote the numbers of boundary vertices drawn in these regions. Finally, let $X_{n}(q)$ be the rank-generating function for $EP_{n}$, that is, the polynomial in $q$ such that the coefficient of $q^{r}$ is $|EP_{n,r}|$. Then, by M\"{o}bius Inversion, we get:

\begin{proposition}
\begin{equation}
X_{n}(q)=\sum_{\pi\in NC_{n}}\left((-1)^{n-k_{\pi}}\prod_{i=1}^{k_{\pi}}C_{\pi_{i}-1}\prod_{j=1}^{n+1-k_{\pi}}T_{a_{j}}(q)\right),
\end{equation}
where $k_{\pi},a_{\pi,j}$ are as before, $\pi_{i}$ denotes the number of elements in the $i$-th block of $\pi$, and the polynomial $T_{m}(q)$ is as in (\ref{crossinggf}).
\end{proposition}

Let us also mention a formula for the bivariate generating function $\sum_{n}X_{n}(q)t^{n}$, whose $q^{r}t^{n}$-coefficient is $|EP_{n,r}|$. Recalling, from \cite{oeis}, the formula
\begin{equation*}
\sum_{n}X_{n}t^{n}=t/F^{\langle -1\rangle}(t),
\end{equation*}
where $F(t)=t\sum_{n}(2n-1)!!t^{n}$ is a shift of the generating function for the sequence of double factorials and $F^{\langle -1\rangle}$ denotes its formal inverse. We may then replace $F(t)$ with the bivariate generating function $F(t,q)=t\sum_{n,r}T_{n}(q)t^{n}$, and obtain
\begin{equation*}
\sum_{n}X_{n}(q)t^{n}=t/F^{\langle -1\rangle}(t,q),
\end{equation*}
where here the inverse is taken with respect to $t$ only.

To conclude this section, we cannot resist making the following conjecture:

\begin{conjecture}\label{unimodal}
$EP_{n}$ is rank-unimodal.
\end{conjecture}

In support of Conjecture \ref{unimodal}, let us list the rank sizes of $EP_{n}$ below, for small values of $n$.

\begin{tabular}{ |l|l| }
\hline
\multicolumn{2}{ |c| }{Rank Sizes} \\
\hline
$n$ & Rank Size \\ \hline
1 & 1 \\ \hline
2 & 1, 1 \\ \hline
3 & 1, 3, 3, 1 \\ \hline
4 & 1, 6, 14, 16, 10, 4, 1 \\ \hline
5 & 1, 10, 40, 85, 110, 97, 65, 35, 15, 5, 1, \\ \hline
6 & 1, 15, 90, 295, 609, 873, 948, 840, 636, 421, 246, 126, 56, 21, 6, 1 \\ \hline
\multirow{2}{*}{7} & 1, 21, 175, 805, 2366, 4872, 7567, 9459, \\
& 10031, 9359, 7861, 6027, 4249, 2765, 1661, 917, 462, 210, 84, 28, 7, 1 \\ \hline
\multirow{2}{*}{8} & 1, 28, 308, 1876, 7350, 20272, 42090, 69620, 96334, 115980, 125044, 123176, 112380, 95836, \\
& 76868, 58220, 41734, 28344, 18236, 11096, 6364,
3424, 1716, 792, 330, 120, 36, 8, 1 \\
\hline
\end{tabular}

\section{Acknowledgments}

This work was undertaken at the REU (Research Experiences for Undergraduates) program at the University of Minnesota-Twin Cities, supported by NSF grants DMS-1067183 and DMS-1148634. The authors thank Joel Lewis, Gregg Musiker, Pavlo Pylyavskyy, and Dennis Stanton for their leadership of the program, and are especially grateful to Joel Lewis and Pavlo Pylyavskyy for introducing them to this problem and for their invaluable insight and encouragement. The authors also thank to Thomas McConville for many helpful discussions. Finally, the authors thank Vic Reiner, Jonathan Schneider, and Dennis Stanton for suggesting references, and Damien Jiang and Ben Zinberg for formatting suggestions.

\appendix

\section{Proofs of Lemmas \ref{dbound} and \ref{ebound}}\label{bounds}

Recall the definitions of $D_{n},E_{n},X_{n}$ from \S \ref{asymptoticsection}. We will prove Lemmas \ref{dbound} and \ref{ebound}, that $D_{n}/X_{n}\to\sqrt{e}-1$ and $E_{n}/X_{n}\to0$, respectively.

\begin{proof}[Proof of Lemma \ref{dbound}]
We may as well consider $D_{n}-1=\sum_{j=1}^{n-2}\binom{n}{j}X_{n-j}$. Using the notation $Q_{i}=X_{i}/X_{i-1}$, as in the proof of Lemma \ref{ratio}, we have
\begin{align*}
\frac{\sum_{j=1}^{n-2}\binom{n}{j}X_{n-j}}{X_{n}}&=\sum_{j=1}^{n-2}\frac{1}{j!}\cdot\frac{n(n-1)\cdots(n-j+1)}{Q_{n}Q_{n-1}\cdots Q_{n-j+1}}\\
&=\sum_{j=1}^{n-2}\frac{1}{2^{j}j!}\cdot\frac{2n(2n-2)\cdots(2n-2j+2)}{Q_{n}Q_{n-1}\cdots Q_{n-j+1}}\\
&=\sum_{j=1}^{n-2}\frac{1}{2^{j}j!}+\sum_{j-1}^{n-2}\frac{1}{2^{j}j!}\left(\frac{2n(2n-2)\cdots(2n-2j+2)}{Q_{n}Q_{n-1}\cdots Q_{n-j+1}}-1\right).
\end{align*}
As $n\rightarrow\infty$, first summand above converges to $\sqrt{e}-1$, so it is left to check that the second summand converges to zero.

Note that, by Lemma \ref{ratio},
\begin{equation*}
\frac{2n(2n-2)\cdots(2n-2j+1)}{Q_{n}Q_{n-1}\cdots Q_{n-j+2}}>1.
\end{equation*}
Now,
\begin{align*}
0&<\sum_{j=1}^{n-2}\frac{1}{2^{j}j!}\left(\frac{2n(2n-2)\cdots(2n-2j+2)}{Q_{n}Q_{n-1}\cdots Q_{n-j+1}}-1\right)\\
&<\sum_{j=1}^{n-5}\frac{1}{2^{j}j!}\left(\frac{2n(2n-2)\cdots(2n-2j+2)}{Q_{n}Q_{n-1}\cdots Q_{n-j+1}}-1\right)\\
&\qquad+Kn\left[\frac{1}{2^{n-4}(n-4)!}+\frac{1}{2^{n-3}(n-3)!}+\frac{1}{2^{n-2}(n-2)!}\right],
\end{align*}
for some positive constant $K$, because 
\begin{align*}
&\frac{2n(2n-2)\cdots(2n-2j+2)}{Q_{n}Q_{n-1}\cdots Q_{n-j+1}}\\
<\ &2n\cdot\frac{2n-2}{Q_{n}}\cdot\frac{2n-4}{Q_{n-1}}\cdots\frac{2n-2j+2}{Q_{n-j+2}}\cdot\frac{1}{Q_{n-j+1}}\\
<\ &Kn,
\end{align*}
as by Lemma \ref{ratio}, all but a fixed number of the fractions are less than 1, and those which are not are constant. It is then easy to see that the term 
\begin{equation*}
Kn\left[\frac{1}{2^{n-4}(n-4)!}+\frac{1}{2^{n-3}(n-3)!}+\frac{1}{2^{n-2}(n-2)!}\right]
\end{equation*}
goes to zero as $n\rightarrow\infty$. Now, applying Lemma \ref{ratio} again (noting that the indices are all at least 6),
\begin{align*}
&\sum_{j=1}^{n-5}\frac{1}{2^{j}j!}\left(\frac{2n(2n-2)\cdots(2n-2j+2)}{Q_{n}Q_{n-1}\cdots Q_{n-j+1}}-1\right)\\
<\ &\sum_{j=1}^{n-5}\frac{1}{2^{j}j!}\left(\frac{2n(2n-2)\cdots(2n-2j+2)}{(2n-1)(2n-3)\cdots(2n-2j+1)}-1\right)\\
<\ &\sum_{j=1}^{n-5}\frac{1}{2^{j}j!}\cdot\left(\frac{2n}{2n-2j+1}-1\right)\\
<\ &\sum_{j=1}^{n-5}\frac{1}{2^{j}j!}\cdot\frac{2j-1}{2n-2j+1}\\
<\ &\sum_{j=1}^{n-5}\frac{1}{2^{j-1}(j-1)!}\cdot\frac{1}{2n-2j+1}.
\end{align*}
It is enough to show that the above sum goes to zero as $n\rightarrow\infty$. To do this, we split it in to two sums:
\begin{align*}
&\sum_{j=1}^{n-5}\frac{1}{2^{j-1}(j-1)!}\cdot\frac{1}{2n-2j+1}\\
=\ &\sum_{1\le j<n/2}\frac{1}{2^{j-1}(j-1)!}\cdot\frac{1}{2n-2j+1}+\sum_{n/2\le j\le n-5}\frac{1}{2^{j-1}(j-1)!}\cdot\frac{1}{2n-2j+1}\\
<\ &\sum_{1\le j<n/2}\frac{1}{2^{j-1}(j-1)!}\cdot\frac{1}{n}+\sum_{n/2\le j\le n-5}\frac{1}{2^{j-1}(j-1)!}\\
<\ &\frac{\sqrt{e}}{n}+\sum_{n/2\le j\le n-5}\frac{1}{2^{j-1}(j-1)!}.
\end{align*}
The first summand clearly tends to zero as $n\rightarrow\infty$. The rest of the sum must tend to zero as well, as it is the tail of a convergent sum, so the proof is complete.
\end{proof}

\begin{proof}[Proof of Lemma \ref{ebound}]
First, note that by Corollary \ref{limit}, $X_{i}$ is within a (positive) constant factor of $(2i-1)!!$ for each $i$. Thus, to prove that $E_{n}/X_{n}\to0$, we may as well prove that 
\begin{equation*}
n\sum_{j=2}^{n-2}\frac{(2j-1)!!(2n-2j-1)!!}{(2n-1)!!}\to0.
\end{equation*}
It is straightforward to check that the largest terms of the sum are when $j=2,n-2$, and these terms are of inverse quadratic order. Thus, 
\begin{equation*}
n\sum_{j=2}^{n-2}\frac{(2j-1)!!(2n-2j-1)!!}{(2n-1)!!}<nO(n^{-2})=O(n^{-1}),
\end{equation*}
and the conclusion follows.
\end{proof}

\end{document}